\documentclass[11pt,a4paper]{amsart}

\usepackage{amsmath, amssymb}
\usepackage{amsthm}
\usepackage{amscd}
\usepackage{color}
\usepackage{tikz}
\usepackage[all,cmtip]{xy}
\usepackage{mathtools}
\usepackage{lscape}
\usepackage{rotating}
\usepackage{mathdots}
\usepackage{hyperref}
\usepackage{enumitem}

\usepackage[margin=1in, marginparwidth=0.8in]{geometry}   

\newcommand{\CC}{\mathbb{C}}

\newcommand{\ZZ}{\mathbb{Z}}
\newcommand{\QQ}{\mathbb{Q}}

\newcommand{\cC}{\mathcal{C}}

\DeclareMathOperator{\Hom}{Hom}
\DeclareMathOperator{\ind}{ind}

\newtheorem{thm}{Theorem}[section]
\newtheorem{theorem}{Theorem}[section]  
\newtheorem{prop}[thm]{Proposition}
\newtheorem*{propos}{Proposition}
\newtheorem{cor}[thm]{Corollary}
\newtheorem{lem}[thm]{Lemma}
\newtheorem{conj}[thm]{Conjecture}

\theoremstyle{remark}
\newtheorem{remark}[thm]{Remark}
\newtheorem{ex}[thm]{Example}
\newtheorem{notation}[thm]{Notation}
\newtheorem{definition}[thm]{Definition}
\newtheorem{corollary}[thm]{Corollary}

\newcommand{\Amr}{A_{\underline{m};r}}
\newcommand{\bmr}{b_{\underline{m};r}}

\DeclareMathOperator{\sgn}{sgn}
\DeclareMathOperator{\SL}{SL}

\author[K.~Baur, E.~Faber, S.~Gratz, K.~Serhiyenko, G.~Todorov]{Karin Baur, Eleonore Faber, Sira Gratz, Khrystyna Serhiyenko, Gordana Todorov}
\address{School of Mathematics, University of Leeds, Leeds, LS2 9JT, UK. On leave from 
the University of Graz, Austria}
\email{pmtkb@leeds.ac.uk}

\address{
School of Mathematics, University of Leeds, Leeds, LS2 9EJ, UK
}

\email{e.m.faber@leeds.ac.uk}

\address{
School of Mathematics and Statistics, University of Glasgow, University Place, Glasgow, G12 8SQ, UK
}
\email{Sira.Gratz@glasgow.ac.uk}

\address{Department of Mathematics, University of California at Berkeley, Berkeley, CA 94720, USA}
\email{khrystyna.serhiyenko@berkeley.edu}

\address{Department of Mathematics, Northeastern University, Boston, MA 02115, USA}
\email{g.todorov@northeastern.edu}

\date{\today}
\thanks{ 
\noindent 
K.B. was supported by FWF grants P 30549-N26 and W1230. She 
is supported by a Royal Society Wolfson Research Merit Award.
E.F. is a Marie Sk{\l}odowska-Curie fellow at the University of Leeds 
(funded by the European Union's Horizon 2020 research and innovation programme under the 
Marie Sk{\l}odowska-Curie grant agreement No 789580).
K.S. was supported by NSF Postdoctoral Fellowship MSPRF - 1502881.
} 

\title{Friezes satisfying higher SL$_k$-determinants}

\begin{document}

\baselineskip=16pt

\maketitle

\begin{abstract} 
In this article, we construct $\SL_k$-friezes using Pl\"ucker coordinates, 
making use of the cluster structure on the homogeneous coordinate ring of the
Grassmannian of $k$-spaces in $n$-space via the Pl\"ucker embedding. 
When this cluster algebra is of finite type, the
$\SL_k$-friezes are in bijection with the so-called mesh friezes of the 
corresponding Grassmannian cluster category. These are collections of positive integers on the AR-quiver of 
the category with relations inherited from the mesh relations on the category.
In these finite type cases, many of the $\SL_k$-friezes arise from specialising a cluster to 1. 
These are called unitary. 
We use Iyama-Yoshino reduction to analyse the non-unitary friezes. With this, 
we 
provide an explanation for all known friezes of this kind. 
An appendix by Cuntz and Plamondon proves that there are 868 
friezes of type $E_6$. 
\end{abstract}

\begin{small}
\setcounter{tocdepth}{1}
\tableofcontents
\end{small}

%
\section{Introduction} 

In this paper, we establish an explicit connection between 
SL$_k$-friezes and Grassmannian cluster categories and algebras.

Integral $\SL_k$-friezes are certain arrays of integers consisting of finitely many rows of infinite length, see Example~\ref{ex:SL2}.  Moreover, entries in an $\SL_k$ frieze satisfy the so-called diamond rule, where for every $k\times k$ diamond formed by the neighboring entries, the determinant of the corresponding matrix equals 1.  For example, when $k=2$ each diamond $\begin{smallmatrix}&a\\b&&c\\&d\end{smallmatrix}$ satisfies the relation $\big|{\begin{smallmatrix}b&a\\d&c\end{smallmatrix}}\big|=1$.   Moreover,  we consider tame $\SL_k$ frieze, that is friezes where the determinant of every $(k+1)\times(k+1)$ diamond is 0.   Such friezes have horizontal period $n$, where $n$ is determined by $k$ and the number of rows.

$\SL_k$-friezes were introduced in the seventies by Coxeter \cite{Coxeter} in the case $k=2$, while higher $\SL_k$-frieze patterns first appeared in work of Cordes-Roselle \cite{CordesRoselle}.   Conway and Coxeter further studied $\SL_2$-friezes in \cite{CoCo1, CoCo2}, where they showed that there exists a bijection between $\SL_2$ friezes and triangulations of polygons.  Interest in friezes and their various generalizations renewed after the introduction of cluster algebras in 2001, as cluster algebras coming from the Grassmannian of 2 planes in an $n$-dimensional space are also in bijection with triangulations of polygons.  Moreover, it was shown later that $\SL_2$ friezes can be obtained by specializing all cluster variables in a given cluster to one \cite{CalderoChapoton}.

In this way, $\SL_2$ friezes are well-understood and they are closely related to the combinatorics of cluster algebras.    On the other hand, the classification of integral $\SL_k$-friezes remains elusive.  Our paper makes a step in the direction of a complete classification.  We show that in the finite type cases, all integral $\SL_k$-friezes can be obtained from the combinatorics of the cluster algebras on coordinate rings of Grassmannians, and Grassmannian cluster categories.  We make use of the combinatorial tools that we call 
\emph{Pl\"ucker friezes}, which arise from the cluster algebras, and \emph{mesh friezes}, which arise from Grassmannian cluster categories.  

Pl\"ucker friezes play a crucial role in our construction of $\SL_k$-friezes.   Their entries are given by a constellation of Pl\"ucker coordinates in the homogeneous coordinate ring $\mathcal{A}(k,n)$ of the Grassmannian $\mathrm{Gr}(k,n)$ via the Pl\"ucker embedding.
We show that the (specialized) Pl\"ucker frieze of type $(k,n)$ deserves its name: it is indeed an $\SL_k$-frieze -- all its $k \times k$ diamonds have determinant 1 (cf.\ Theorem \ref{thm:frieze}).  In particular, each map from $\mathcal{A}(k,n)$ to $\mathbb{Z}$, given by a specialization of a cluster to one, yields a tame integral $\SL_k$-frieze (cf.\ Corollary \ref{cor:cluster-1-frieze}).   We call friezes obtained in this way \emph{unitary}.  This generalizes the results about friezes from the case 
$k=2$, which have been well-known and studied for some time: Indeed, all $\SL_2$-friezes arise from the (specialized) Pl\"ucker frieze of type $(2,n)$ in this way.  However,  for $k>2$ not all friezes are unitary.

Next, we use categorification of cluster algebras to obtain all $\SL_k$ friezes whenever $\mathcal{A}(k,n)$ is of finite type.  A mesh frieze is an integral frieze on the Auslander-Reiten quiver of a Grassmannian cluster category such that entries coming from a mesh satisfy a certain frieze-like relation.
In the case $k=2$, 
a mesh frieze of the cluster category exactly agrees with the $\SL_2$-frieze. Thus, in light of the connection that 
we found, it makes sense to consider relationship between $\SL_k$-friezes and 
mesh friezes of Grassmannian cluster categories for $k>2$.  In finite type we show that there exists a bijection between the two.  In particular, this implies that in finite type every $\SL_k$ frieze is obtained by specializing the collection of cluster variables in a given cluster to some set of positive integers.

This also allows us to use tools from representation theory, such as Iyama-Yoshino reduction, to study $\SL_k$-friezes. It enables us to pass from a mesh frieze of a given ``rank'' (that is, a mesh frieze for a cluster category whose rank is given by the number of indecomposables in a cluster tilting object), to one of lower rank. Making use of known restrictions for smaller rank cases (in particular, of type $A$ cases) we can then draw conclusions about the nature of our original mesh frieze. 
The natural bijection, in finite type cases, between mesh friezes and $\SL_k$-friezes means that Iyama-Yoshino reduction helps us better our understanding of integral $\SL_k$-friezes.
In particular, we obtain new results on the number of such friezes and their possible entries. 

An appendix by Cuntz and Plamondon determines the number of friezes in type $(3,7)$. 

\subsection*{Acknowledgements} We want to thank Hugh Thomas for comments on an earlier version of the paper. We also want to thank the two anonymous referees for their helpful comments and suggestions.

%
\section{Background and notation}

%
\subsection{$\SL_k$-friezes}

\ 

\begin{definition}\label{D:frieze}
Let $R$ be an integral domain. \\
An {\em $\SL_k$-frieze (over $R$)} is a 
grid in the plane, consisting of a finite number of rows:  
\[
\xymatrix@C=.08em@R=.06em{
 & 0 && 0 && 0 && 0 && 0 && 0 && 0  \\ 
\dots && \iddots&& \iddots &   & \iddots  && \iddots && \iddots && \iddots && \dots\\ 
& 0 && 0 && 0 && 0 && 0 && 0 && 0  \\
\dots & & 1 && 1 && 1 && 1 && 1 && 1 &&  \dots \\
 & \ast && \ast && \ast &&m_{0,w-1} && m_{1,w} &&  m_{2,w+1} && \ast \\
 &  &  &&  && && \iddots \\
 & \dots && \dots && m_{01} && m_{12} && m_{23} && m_{34} && \dots \\
 &  & \dots && m_{00} && m_{11} && m_{22} && m_{33} && \dots  && \dots\\
 & 1 && 1 && 1 && 1 && 1 && 1 && 1 && \\ 
\dots &  & 0 && 0 && 0 && 0 && 0 && 0 &&  \dots \\
& && \iddots&& \iddots &   & \iddots  && \iddots && \iddots &&  \\
\dots &  &0 & & 0 && 0 &&  0 && 0 && 0 &&  &&  \\
}
\]
The frieze is formed by $k-1$ rows of zeroes followed by a row of 1s from top and from 
bottom respectively 
and by $w\ge 1$ rows of elements $m_{ij}\in R$ in between, 
such that every $k\times k$-diamond of entries of the frieze has determinant 1, i.e.~,
whenever we consider a matrix having 
$k$ successive entries of a row of an $\SL_k$-frieze on its diagonals and all other entries above 
and below the diagonal accordingly, the determinant of this matrix is $1$. The integer $w$ is called the \emph{width} of the frieze.
\end{definition}

\begin{definition} \label{def:SLk-frieze}
Let $F$ be an $\SL_k$-frieze over $R$. \\
(1) $F$ is {\em tame} if all $(k+1)\times (k+1)$-diamonds in $F$ have determinant 0. In general, we define an $s \times s$-diamond of the frieze, where $1 \leq s \leq k+1$, as the matrix having $s$ successive entries of the frieze on its diagonals and all other entries above and below the diagonal accordingly.\\
(2) If all non-trivial entries $m_{ij}$ with $0 \leq i,j \leq w$ of $F$ are positive integers, we call $F$ an {\em integral} 
$\SL_k$-frieze. 
\end{definition}

\begin{remark}\label{rem:width-period}
One can prove that every tame $\SL_k$-frieze of width $w$ over a field $K$
has horizontal period $n=w+k+1$, which was first proven for certain $\SL_k$-friezes in 
\cite{CordesRoselle} and in general in \cite[Cor.~7.1.1]{mgost14}. 
\end{remark}

$\SL_2$-friezes were first studied by Coxeter and Conway--Coxeter in the early 1970s, see Example \ref{ex:SL2} below. Higher $\SL_k$-frieze patterns made their first appearance 1972 in work of Cordes--Roselle \cite{CordesRoselle} (with extra conditions on the minors of the first  $k \times k$-diamonds) and seem only to have re-emerged with the introduction of cluster algebras: as $2$-friezes \cite{Propp}, $\SL_k$-tilings in \cite{BergeronReutenauer} and more systematically in \cite{MGOT12, mgost14, mg12, Cuntz}. See also \cite{Crossroads} for a survey on different types of frieze patterns.

\begin{ex}\label{ex:SL2}
An $\SL_2$-frieze is simply called a {\em frieze pattern}. 
Frieze patterns have first been studied by Coxeter, \cite{Coxeter} 
and by Conway and Coxeter in \cite{CoCo1, CoCo2}. They showed that 
integral frieze patterns of width $w$ are in bijection with triangulations of 
$w+3$-gons. Their horizontal period is $w+3$ or a divisor of $w+3$. The following 
example 
\[
\includegraphics[width=12cm]{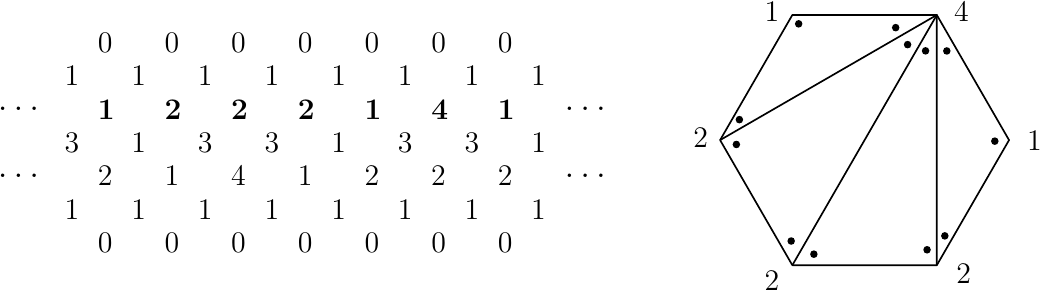}
\]
arises from a fan triangulation of a hexagon: 
The first non-trivial row of the frieze pattern (in bold face) is given by the 
number of triangles of the triangulation meeting at the vertices of the hexagon
(counting these while going around the hexagon). 
\end{ex}

%
\subsection{Pl\"ucker relations}\label{sec:pluecker-new}
 
\

Throughout, we fix $k,n \in \ZZ_{>0}$ with $2 \leq k \leq \frac{n}{2}$ and consider the 
Grassmannian $\mathrm{Gr}(k,n)$ as a projective variety via the Pl\"ucker embedding, 
with homogeneous coordinate ring 
\[
\mathcal{A}(k,n) = \CC[\mathrm{Gr}(k,n)].
\] 

This coordinate ring can be equipped with the structure of a cluster algebra as we 
briefly recall here. For details, we refer to~\cite{Scott06}. 

We write $[1,n] = \{1,2, \ldots, n\}$ for the closed interval of integers between $1$ and $n$. 
A $k$-subset of $[1,n]$ is a $k$-element subset of $[1,n]$.  
Every $k$-subset $I=\{i_1,\dots, i_k\}$ with $i_1<\dots< i_k$ of $[1,n]$ gives rise to 
a {\em Pl\"ucker coordinate}  
$p_{i_1,i_2,\dots, i_k}$.
We extend this definition to allow different ordering on the indices and repetition of 
indices: for an arbitrary $k$-tuple $I=(i_1,\dots, i_k)$ 
we set $p_{i_1,\dots, i_k}=0$ if there exists $\ell\ne m$ such that 
$i_{\ell}=i_m$ and 
$p_{i_1,\dots, i_k}=\sgn(\pi)p_{j_1,\dots,j_k}$ 
if $(i_1,\dots,i_k)=(j_1,\dots, j_k)$, with $j_1<j_2<\dots<j_k$ and 
$\pi$ is the permutation sending 
$j_m$ to $i_m$ for $1\le m\le k$.  
Sometimes, by abuse of notation, we will also call any tuple $(i_1, \ldots, i_k)$ a $k$-subset.
Scott proved in~\cite{Scott06} that $\mathcal{A}(k,n)$ is a cluster algebra, where all 
the Pl\"ucker coordinates are cluster variables.  For each $(k,n)$, there exists a cluster consisting of 
Pl\"ucker coordinates and the exchange relations arise from the Pl\"ucker relations (see below). 

Note that throughout the article 
we will consider sets $I$ up to reduction modulo $n$, where we choose the representatives 
$1, \ldots, n$ for the equivalence classes modulo $n$. 
If $I$ consists of $k$ consecutive elements, up to reducing modulo $n$, we call 
$I$ a \emph{consecutive $k$-subset} and 
$p_I$ a {\em consecutive Pl\"ucker coordinate}. These are frozen cluster variables or 
coefficients, 
and we will later set all of them equal to $1$. 
Of importance for our purposes are the Pl\"ucker coordinates which arise from $k$-subsets 
of the form 
$I=I_0\cup \{m\}$ where $I_0$ is a consecutive $(k-1)$-subset of $[1,n]$ and $m$ is any entry 
of $[1,n]\setminus I_0$. We call such $k$-subsets {\em almost consecutive}. 
Note that each consecutive $k$-subset is almost consecutive. In particular the 
frozen cluster variables are consecutive and 
thus 
almost consecutive.

We follow the exposition of~\cite[Section 9.2]{Marsh13} for the presentation of the 
Pl\"ucker relations. 
In the cluster algebra $\mathcal A(k,n)$, the relations 
\begin{eqnarray}\label{eq:plucker-new}
\sum_{\ell=0}^k(-1)^l p_{i_1,\dots, i_{k-1},j_{\ell}} p_{j_0,\dots, \widehat{j_{\ell}},\dots, j_k}=0 
\end{eqnarray}
hold for arbitrary 
$1\le i_1<i_2<\dots <i_{k-1}\le n$ and $1\le j_0<j_1<\dots <j_k\le n$, where $\,\widehat{\cdot}\,$ signifies omission. These relations 
are called the {\em Pl\"ucker relations}. 
For $k=2$, the non-trivial Pl\"ucker relations are of the form 
\begin{eqnarray}\label{eq:plucker-k-2}
p_{a,c}p_{b,d} = p_{a,b}p_{c,d} + p_{a,d}p_{b,c} 
\end{eqnarray}
for arbitrary $1\le a<b<c<d\le n$. 
They are often called the three-term Pl\"ucker relations. 

Obviously, the order of the elements in a tuple plays a significant role. We will often need 
to use ordered tuples, or partially ordered tuples. 
For this, we introduce the following notation. For a tuple $J=\{i_1,\dots, i_{\ell}\}$ 
of $[1,n]$ and $\{a_1, \ldots, a_\ell\} = \{i_1, \ldots, i_\ell\}$ with $1\le a_1\le \dots\le a_{\ell}\le n$ 
we set 
\[
o(J)=o(i_1,\dots, i_{\ell}) = (a_1,\dots, a_{\ell}).
\]
With this notation, if $0\le s,\ell\le k$, we will write $p_{b_1,\dots, b_s,o(i_1,\dots, i_{\ell}),  
c_1,\dots, c_{k-\ell+a}}$ for \newline $p_{b_1,\dots, b_s,a_1,\dots, a_{\ell}, 
c_1,\dots, c_{k-\ell+a}}$. 
For arbitrary $I=\{i_1,\dots, i_{k-1}\}$ and $J=\{j_0,\dots, j_k\}$, the Pl\"ucker 
relations (\ref{eq:plucker-new}) then take the form 
 \begin{eqnarray*}
	\sum_{l=0}^k (-1)^{\ell} p_{o(I) j_{\ell}}\cdot p_{o(J\setminus j_{\ell})} =0.
\end{eqnarray*}

%
\subsection{Pl\"ucker friezes, and special Pl\"ucker friezes} 

\ 

To the homogeneous coordinate ring $\mathcal{A}(k,n)$ we now associate a certain frieze pattern 
which we will use later to obtain $\SL_k$-friezes. 
It uses the almost consecutive Pl\"ucker coordinates. For brevity, it is convenient to write 
$[r]^{\ell}$ for the set $\{r,r+1,\dots, r+\ell-1\}$. For example, if $k=4$ and $n=9$, 
$p_{[1]^{3},6}$ is the short notation for $p_{1236}$ and $p_{o([8]^3,4)}$ is short for 
$p_{1489}$. 

\begin{definition}\label{D:Plucker frieze}
	The {\em Pl\"ucker frieze} of type $(k,n)$ is a $\ZZ \times \{1, 2, \ldots, n+k-1\}$-grid with entries given by the map
	\begin{eqnarray*}
		\varphi_{(k,n)} \colon \ZZ \times \{1, 2, \ldots, n+k-1\} & \to & \mathcal{A}(k,n)\\
		(r,m) & \mapsto & p_{o([r']^{k-1},m')},
	\end{eqnarray*}
	where $r'\in[1,n]$ is the reduction of $r$ modulo $n$	and $m'\in[1,n]$ is the 
	reduction of $m+r'-1$ modulo $n$.  
We denote the Pl\"ucker frieze by $\mathcal P_{(k,n)}$. 
\end{definition}

Example~\ref{ex:case2-5} and Figure~\ref{fig:frieze with plucker coordinates} 
illustrate the map $\varphi_{(k,n)}$. 

\begin{ex}\label{ex:case2-5}
We draw $\mathcal P_{(2,5)}$ as grid of Pl\"ucker coordinates 
with the positions in $\ZZ\times \{1,2,\dots, 6\}$ written in grey above them: 
\[
\xymatrix@C=-8pt@R=-1pt{
 && && && && && && && && \\
{\textcolor{blue}{_{m=6}}}\ar@{..>}[rrrrrrrrrrrrr]  &&  \textcolor{gray}{(\text{-}2,6)} 
 &&  \textcolor{gray}{(\text{-}1,6)} &&  \textcolor{gray}{(0,6)} && 
  \textcolor{gray}{(1,6)} &&  \textcolor{gray}{(2,6)} &&  \textcolor{gray}{(3,6)} & \\ 
 &&  p_{33} && p_{44} && p_{55} && p_{11} && p_{22} && p_{33}  \\
  & \textcolor{gray}{(\text{-}2,5)} &&  \textcolor{gray}{(\text{-}1,5)} &&  \textcolor{gray}{(0,5)} 
   &&  \textcolor{gray}{(1,5)} &&  \textcolor{gray}{(2,5)} &&  \textcolor{gray}{(3,5)} 
    &&  \textcolor{gray}{(4,5)} \\ 
   & p_{23} && p_{34} && p_{45} && p_{15} && p_{12} && p_{23} && p_{34} \\
\dots &&  \textcolor{gray}{(\text{-}1,4)} &&  \textcolor{gray}{(0,4)} &&  \textcolor{gray}{(1,4)} 
  &&  \textcolor{gray}{(2,4)} &&  \textcolor{gray}{(3,4)} &&  \textcolor{gray}{(4,4)} && \dots \\ 
 && p_{24} && p_{35} && p_{14} && p_{25} && p_{13} && p_{24}   \\
&   \textcolor{gray}{(\text{-}1,3)} &&  \textcolor{gray}{(0,3)} &&  \textcolor{gray}{(1,3)} && 
  \textcolor{gray}{(2,3)} &&  \textcolor{gray}{(3,3)} &&  \textcolor{gray}{(4,3)} && 
   \textcolor{gray}{(5,3)} \\ 
 & p_{14} && p_{25} && p_{13} && p_{24} && p_{35} && p_{14} && p_{25} \\
{\textcolor{blue}{_{m=2}}}\ar@{..>}[rrrrrrrrrrrrr] &&  \textcolor{gray}{(0,2)} && 
 \textcolor{gray}{(1,2)} &&  \textcolor{gray}{(2,2)} && 
  \textcolor{gray}{(3,2)} &&  \textcolor{gray}{(4,2)} &&  \textcolor{gray}{(5,2)} &\\ 
 && p_{15} && p_{12} && p_{23} && p_{34} && p_{45} && p_{15} \\
&  \textcolor{gray}{(0,1)} &&  \textcolor{gray}{(1,1)}  &&  \textcolor{gray}{(2,1)} &&  
\textcolor{gray}{(3,1)} &&  \textcolor{gray}{(4,1)} &&  \textcolor{gray}{(5,1)} && 
  \textcolor{gray}{(6,1)}  \\
& p_{55} && p_{11} && p_{22} && p_{33} && p_{44} && p_{55} && p_{11} \\
&& \textcolor{red}{_{r=1}}\ar@{..>}[ruuruuruuruuruuruuu]
}
\]
\end{ex}

\begin{remark}
Notice that if $m' \notin \{1, 2, \ldots, k-1\}$ 
then the 
image $\varphi_{(k,n)}(r,m)$ corresponds to a Pl\"ucker coordinate 
$p_{o([r']^{k-1},m')}\ne 0$ and for 
$m' \in \{1, 2, \ldots, k-1\}$
, the image is 0. 
In particular, the top and bottom $k-1$ rows of $\mathcal P_{(k,n)}$ 
consist of solely $0$ entries.
\end{remark}

\begin{landscape}
\begin{figure}
\begin{tikzpicture}[scale = 1.12]
\tikzstyle{every node}=[font=\small]

\foreach \x in {1}
    \foreach \y in {1,...,2} 
      	\pgfmathsetmacro\result{\x +\y -1}
       \node[blue] (\x,\y) at (\x+\y+\y+1,\x + 1) {$p_{o([\y]^{k-1},\pgfmathprintnumber{\result} )}$};
       
       \node[blue] (k-1,1) at (4+2,4) {$p_{o( [1]^{k-1}, k-1)}$};
       \node[blue] (k,2) at (4+4,4) {$p_{o([2]^{k-1}, k )}$};
       
       \node[red] (k,1) at (5+2,5) {$p_{o( [1]^{k-1}, k )}$};
       \node[red] (k,2) at (5+4,5) {$p_{o( [2]^{k-1}, k+1 )}$};
       
       \node (k+1,1) at (6+2,6) {$p_{o( [1]^{k-1}, k+1 )}$};
       \node (k+2,2) at (6+4,6) {$p_{o( [2]^{k-1}, k+2 )}$};
       
       \node (n-1, 1) at (8+2,8) {$p_{o( [1]^{k-1}, n-1 )}$};
       \node (n, 2) at (8+4,8) {$p_{o( [2]^{k-1}, n )}$};
       
       \node[red] (n, 1) at (9+2,9) {$p_{o( [1]^{k-1}, n )}$};
       \node[red] (n+1, 2) at (9+4,9) {$p_{o( [2]^{k-1}, 1 )}$};
       
       \node[blue] (n+1, 1) at (10+2,10) {$p_{o( [1]^{k-1}, 1 )}$};
       \node[blue] (n+2, 2) at (10+4,10) {$p_{o( [2]^{k-1}, 2 )}$};

       \node[blue] (n-1, 1) at (12+2,12) {$p_{o( [1]^{k-1}, k-1 )}$};
       \node[blue] (n, 2) at (12+4,12) {$p_{o( [2]^{k-1}, k )}$};
 
       \foreach \x in {2,4}
    \foreach \y in {3,5} 
       \node[blue] (\x,\y) at (\x+\y+\y,\x) {$\ldots$};

        \foreach \x in {3,11}
    \foreach \y in {1,...,5}
       \node[blue] (\x,\y) at (\x+\y+\y,\x) {$\ldots$};
       
        \foreach \x in {7}
    \foreach \y in {1,...,5}
       \node (\x,\y) at (\x+\y+\y,\x) {$\ldots$};
       
        \foreach \x in {5,9}
    \foreach \y in {3,5} 
       \node[red] (\x,\y) at (\x+\y+\y,\x) {$\ldots$};
       
       \foreach \x in {6,8}
    \foreach \y in {3,5} 
       \node (\x,\y) at (\x+\y+\y,\x) {$\ldots$};
       

    \node[blue] (r,1) at (8 + 2, 2) {$p_{o([r]^{k-1},r )}$}; 
      \node[blue] (k-1,r) at (2+10,4) {$p_{o( [r]^{k-1}, r+k-2 )}$};
      \node[red] (k,r) at (3+10,5) {$p_{o( [r]^{k-1}, r+k-1 )}$};
      \node (k+1,r) at (4+10,6) {$p_{o( [r]^{k-1}, r+k )}$};
      
      \node (n-1, r) at (8+8,8) {$p_{o( [r]^{k-1}, r+n-2 )}$};
      
      \node[red] (n, r) at (9+8,9) {$p_{o( [r]^{k-1}, r+n-1 )}$};
      
      \node[blue] (n+1, r) at (10+8,10) {$p_{o( [r]^{k-1}, r+n )}$};
      
      \node[blue] (top, r) at (12+8,12) {$p_{o( [r]^{k-1}, r+k-2 )}$};
      
      \node (text) at (7,9) {\large{$\ddots$}};
      
      \node (text) at (16,4) {\large{$\ddots$}};
       
\end{tikzpicture}
\caption{The grid defined by $\varphi_{(k,n)}$, reducing indices modulo $n$. 
 The top  and bottom $k-1$ rows (blue) are $0$, the $k$-th row from 
 top/bottom (red) consist of consecutive $k$-tuples, i.e.\ frozen Pl\"ucker coordinates.}
 \label{fig:frieze with plucker coordinates}
\end{figure}
\end{landscape}

\begin{remark}[Specializing coefficients]\label{rem:freeze-coeff}
In what follows, we will replace all the frozen variables by $1$. For that, let 
$J$ be the ideal of $\mathcal A(k,n)$ generated by 
\[
\{x-1\mid x \mbox{ is a consecutive Pl\"ucker coordinate}\}.
\]
The consecutive Pl\"ucker coordinates are precisely the frozen variables in the cluster algebra 
$\mathcal A(k,n)$. The quotient $\mathcal  A(k,n)/J$ is a coefficient-free cluster algebra, 
and therefore -- as a subring of a field of rational functions over $\mathbb{C}$ -- an integral domain.
Let $s: \mathcal A(k,n) \to \mathcal A(k,n)/J$ be the map induced from the 
identity on non-consecutive Pl\"ucker coordinates and from replacing the consecutive 
Pl\"ucker coordinates by $1$. 
This map is an algebra homomorphism $\mathcal A(k,n)\to \mathcal A(k,n)/J$. 
We write $s\mathcal A(k,n)$ for $\mathcal A(k,n)/J$. \\
\end{remark}

\begin{definition}\label{D:Pfb=1}
	The {\em special Pl\"ucker frieze of type $(k,n)$}, denoted by $s\mathcal P_{(k,n)}$, 
	is the grid we get from $\mathcal P_{(k,n)}$ 
	by substituting the consecutive Pl\"ucker coordinates with $1$s. 
	In other words, $s\mathcal P_{(k,n)}$ is the grid with entries given by the map 
	$s \circ \varphi_{(k,n)}$ with entries in $s\mathcal A(k,n)$. 
\end{definition}

\begin{ex}\label{ex:special-2-5}
The special Pl\"ucker frieze of type $(2,5)$, $s\mathcal P_{(2,5)}$, (see Example~\ref{ex:case2-5}) is as follows 
(writing $0$ for the entries $p_{ii}$): 
\[
\xymatrix@C=-8pt@R=-1pt{
&0 && 0 && 0 && 0 && 0 && 0 && 0 && 0\\
  && 1 && 1 && 1 && 1 && 1 && 1 && 1 \\
&p_{13} && p_{24} && p_{35} && p_{14} && p_{25} && p_{13} && p_{24} && \cdots  \\
\cdots && p_{14} && p_{25} && p_{13} && p_{24} && p_{35} && p_{14} && p_{25} \\
& 1  && 1 && 1 && 1 && 1 && 1 && 1 && 1 \\
 && 0 && 0 && 0 && 0 && 0 && 0 && 0
}
\]
And $s\mathcal P_{(3,6)}$: 
\[
\xymatrix@C=-8pt@R=-1pt{
 && 0 && 0 && 0 && 0 && 0 && 0 && 0 \\
&0 && 0 && 0 && 0 && 0 && 0 && 0 && 0\\
  && 1 && 1 && 1 && 1 && 1 && 1 && 1 \\
&p_{146} && p_{125} && p_{236} && p_{134} && p_{245} && p_{356} && p_{146} && \cdots  \\
\cdots && p_{124} && p_{235} && p_{346} && p_{145} && p_{256} && p_{136} && p_{124} \\
& 1  && 1 && 1 && 1 && 1 && 1 && 1 && 1 \\
 && 0 && 0 && 0 && 0 && 0 && 0 && 0 \\
&0 && 0 && 0 && 0 && 0 && 0 && 0 && 0
}
\]
\end{ex}

\begin{remark}
For $k=3$, \cite[Proposition~3.3]{MGOT12} describes an explicit correspondence between $\SL_3$-friezes and 2-friezes 
(for a definition of 2-friezes see Appendix B). It is stated in~\cite[Remark 5.8]{MGOT12} that it can be checked that all the 
$n(n-4)$ entries in a 2-frieze are Pl\"ucker coordinates for the Grassmannian $\mathrm{Gr}(3,n)$. Performing this check and combining the two observations yields an alternative way to obtain the Pl\"ucker frieze of type $(k,n)$ for the special case $k=3$.
More generally, in \cite[Remark~3.19]{Crossroads} it is observed that, for $k \geq 3$, cluster variables for $\mathcal{A}(k,n)$ appear as entries in the derived arrays of an $\SL_k$-frieze.
\end{remark}

%
\section{Integral tame $\SL_k$-friezes from Pl\"ucker friezes}\label{sec:determinant}

%
\subsection{The special Pl\"ucker frieze $s\mathcal P_{(k,n)}$ is a tame $\SL_k$-frieze}

\ 

In this section, we prove the following result. 

\begin{thm} \label{thm:frieze}
	The frieze $s\mathcal P_{(k,n)}$ is a tame $\SL_k$-frieze over $s\mathcal{A}(k,n)$.
\end{thm}

Before we provide a proof for Theorem \ref{thm:frieze}, we make some observations on notation and set-up.

\begin{notation} 
	Throughout we calculate modulo $n$. More precisely, we will always reduce integers 
	modulo $n$, and identify integers with their representatives in $[1,n]$. 
	In addition, the following notation will be useful: For any $a, b \in [1,n]$ we denote by 
	$[a,b]$ the closed interval between $a$ and $b$ modulo $n$, defined as
 \[
   [a,b] = \begin{cases}
   				\{p \in [1,n] \mid a \leq p \leq b\} & \text{if} \; a \leq b\\
   				 \{p \in [1,n] \mid a \leq p \leq n \; \text{or} \; 1 \leq p \leq b\} & \text{if} \; b < a,
			\end{cases}
  \]
 and by $(a,b)$ the open interval between $a$ and $b$ modulo $n$, defined as 
  \[
   (a,b) = \begin{cases}
   				\{p \in [1,n] \mid a < p <b\} & \text{if} \; a \leq b\\
   				 \{p \in [1,n] \mid a < p \leq n \; \text{or} \; 1 \leq p < b\} & \text{if} \; b<a.\\
			\end{cases}
  \]
  Analogously, we define the half-open intervals $[a,b)$ and $(a,b]$. 
\end{notation}

Recall that we obtain the frieze $s\mathcal{P}_{(k,n)}$ from the frieze $\mathcal{P}_{(k,n)}$ 
by specialising the consecutive Pl\"ucker coordinates to $1$. We will compute the determinants 
of a diamond in the frieze $s\mathcal{P}_{(k,n)}$ via the corresponding diamond in $\mathcal{P}_{(k,n)}$. 
The $k \times k$ diamonds in the Pl\"ucker frieze $\mathcal P_{(k,n)}$ 
are matrices of the form
\[
	A_{[m]^k;r} := 
	(p_{o([r+i-1]^{k-1}, m+j-1)})_{1 \leq i,j \leq k}, 
\]
with 
\[
	 r \in [1,n] \; \text{and} \; m \in [r+k-1,r+n-1] = [r+k-1, r-1],
\]
where as before we calculate modulo $n$, with representatives $1, \ldots, n$. 

To inductively compute the determinants of the matrices $A_{[m]^k;r}$, we need to consider other matrices of the following form: 
For $r \in [1,n]$, $1 \leq s \leq n$ 
and $\underline{m}=(m_1, \ldots, m_s)$ with $m_i \in [1,n]$ 
we define 
\[ 
\Amr = (a_{ij})_{1 \leq i,j \leq s} 
\] 
to be the 
$(s \times s)$-matrix with entries 
$a_{ij} = p_{o(  [r+i-1]^{k-1}, m_j)}$ for $1 \leq i,j \leq s$:
  \begin{equation}\label{eq:matrix-def}
  \Amr := \begin{bmatrix*}[l]
  		
				p_{o([r]^{k-1}, m_1)} 		& p_{o([r]^{k-1}, m_2)}	
					& \ldots 	& p_{o([r]^{k-1}, m_s)}\\
				p_{o([r+1]^{k-1}, m_1)} 	& p_{o([r+1]^{k-1}, m_2)}	
				& \ldots 	& p_{o([r+1]^{k-1}, m_s)}\\
				\vdots 			& \vdots 			& \vdots 	& \vdots \\
				p_{o([r+s-1]^{k-1}, m_1)} 	& p_{o([r+s-1]^{k-1}, m_2)} 
				& \ldots 	& p_{o([r+s-1]^{k-1}, m_s)}
			
  \end{bmatrix*}.
\end{equation}
 
\begin{ex}
As an example, let $k=3$ and $n = 6$ and consider the matrix $\Amr$ in $\mathcal{P}_{(3,6)}$ 
for $r = 1$, $s=3$ and $\underline{m} = (3,4,5)$. It is given by
\begin{equation*}\label{eq:matrix-def}
 A_{(3,4,5);1} := \begin{bmatrix*}[l]
  		
				p_{123} 		& p_{124}	 	&p_{125}\\
				p_{233} 		& p_{234}		&p_{235}\\
				p_{334} 		& p_{344}		& p_{345}
			
  \end{bmatrix*}
 = \begin{bmatrix*}[l]
  		
				p_{123} 		& p_{124}	 	&p_{125}\\
				0 		& p_{234}		&p_{235}\\
				0		& 0		& p_{345}
			
 \end{bmatrix*}
\end{equation*}
\end{ex}

We set  
\[
  \bmr := \det  \Amr.
  \]
In particular, for $\underline{m} = [m]^k = (m, m+1, \ldots, m+k-1)$ the matrix $\Amr$ is a $k \times k$ diamond in the frieze $\mathcal{P}_{(k,n)}$ whose determinant it is our first goal to compute. 

We can compute the determinants $\bmr$ provided the two conditions (c1) and (c2) are 
satisfied for the tuple $\underline{m}$: 
\begin{notation}
Fix $r \in [1,n]$ and let $1 \leq s \leq k$. Choose elements $m_j\in[1,n]$ for 
each $1\le j\le s$. 
We can impose the following conditions on the ordered tuple $(m_1,\dots, m_s)$ 
\begin{itemize}
\item[(c1)]{It is ordered cyclically modulo $n$, 
that is there exists a number $b \in \{1, \ldots, s\}$ such that
 \[
 	m_b < m_{b+1} < \ldots < m_s < m_1 < m_2 < \ldots < m_{b-1},
 \]
 if $b \in \{2, \ldots, s\}$ or
 \[
 	m_1 < m_2 < \ldots < m_s
 \]
 if $b=1$.
 }
\item[(c2)]{We have $r+k-2 \notin [m_1,m_s)$.}
\end{itemize}
Conditions (c1) and (c2) are technical conditions needed to ensure that we get the 
correct signs in our computations (cf.\ the proof of Proposition \ref{P:determinant}).
\end{notation}

We now provide a formula for computing the determinants of the matrices $\Amr$ with entries in $\mathcal{A}(k,n)$.
 
 \begin{prop}\label{P:determinant}
 Let $r \in [1,n]$ and let $1 \leq s \leq k$. Let $\underline{m}=(m_1,\dots, m_s)$ 
 with $m_i \in [1,n]$ for all $i$, and assume that $\underline{m}$ satisfies conditions (c1) and (c2).
 Let $\bmr$ be the determinant of the matrix $\Amr$ 
 from (\ref{eq:matrix-def}). 

Then we have
 \begin{eqnarray} \label{E:result}
 	\bmr= \Big[ \prod_{l=0}^{s-2} p_{o([r+l]^k )} \Big] 
	\cdot p_{o( [r+s-1]^{k-s},  m_1,\ldots, m_s)}.
 \end{eqnarray}
\end{prop}

The proof of this proposition can be found in Appendix \ref{Appendix:ProofDet}. 

\begin{remark}
	In particular, if $k=s$ and $\underline{m} = [m]^k$ 
	in the statement of Proposition \ref{P:determinant}, then we have
	\[
		\bmr=b_{m_1,m_2,\ldots,m_k;r} = \Big[\prod_{l=0}^{k-2} p_{o([r+l]^k )} 
		\Big] \cdot p_{o( [m]^k)}; 
	\]
	a product of consecutive Pl\"ucker coordinates.
\end{remark}

	\begin{proof}[Proof of Theorem~\ref{thm:frieze}]
All entries in the bottom $k-1$ rows and in the top $k-1$ rows of 
$s\mathcal P_{(k,n)}$ are zero by definition 
and the entries in the $k$th row from top and from bottom are $1$ since the consecutive 
Pl\"ucker coordinates are set to be $1$, cf.~Definition~\ref{D:Pfb=1}. 

It remains to show that the $k\times k$ determinants in $s\mathcal P_{(k,n)}$ are all 1 and 
that the $(k+1)\times (k+1)$-determinants all vanish. We will use Proposition~\ref{P:determinant} 
for both claims.

Observe that each 
$k \times k$-diamond in the special Pl\"ucker frieze $s\mathcal P_{(k,n)}$ is of the form 
$s(\Amr) = (s(a_{ij}))_{1 \leq i,j \leq k}$, where $\Amr = (a_{ij})_{1 \leq i,j \leq k}$ is a 
$k \times k$-diamond in $\mathcal P_{(k,n)}$ 
(the integer $s$ from Equation \ref{eq:matrix-def} is equal to $k$ here)
and the map $s$ is the 
specialization of consecutive Pl\"ucker coordinates to $1$ (cf.\ Definition \ref{D:Pfb=1}). 
Recall that $\Amr$ is of the form as in (\ref{eq:matrix-def}), where 
\[
	 r \in [1,n] \; \text{and} \; m \in [r+k-1, r-1]
\]
and $\underline{m} = (m,m+1, \ldots, m+k-1)$.
Then the $k$-tuple $\underline{m}$ clearly satisfies Condition (c1). We are going to show that it also satisfies Condition (c2). 
Indeed, if we had $r+k-2 \in [m,m+k-1)$, then $r+k-2 =m+j-1$ for some $1 \leq j < k$. However, since $m \in [r+k-1, r-1]$, this would imply
\[
	r+k-2 = m + j -1 \in [r+k+j-2, r + j -2].
\]
But for $1 \leq j < k$ we have $r+k-2 \in (r+j-2,r+k+j-2)$, and $[r+k+j-2, r+j-2] \cap (r+j-2,r+k+j-2) = \varnothing$; a contradiction.
So we must have $r+k-2 \notin [m,m+k-1)$.

Therefore, both (c1) and (c2) are satisfied by $\underline{m}$ and we can 
apply Proposition \ref{P:determinant} to obtain
		\[
			\mathrm{det}(\Amr) = \bmr = \prod_{l=0}^{k-2} p_{o( [r+l]^k)}  
			\cdot p_{o([m]^k)} .
		\]
		
This yields
\[
 \mathrm{det}(s(\Amr) =  s(\bmr) = \prod_{l=0}^{k-2} s(p_{o([r+l]^k)}) 
 \cdot s(p_{o( [m]^k)})  = 1,
\]
for the $k \times k$ diamond $s(\Amr)$ of the special Pl\"ucker frieze 
$s\mathcal P_{(k,n)}$. Since this holds for every $k \times k$ diamond, 
$s\mathcal P_{(k,n)}$ is indeed a $\SL_k$-frieze.

To show that it is tame, consider an arbitrary $(k+1) \times (k+1)$ diamond of 
$s\mathcal P_{(k,n)}$. It must be of the form $s(A_{[m]^{k+1};r}) = (s(a_{ij}))_{1 \leq i,j \leq k+1}$, 
where $A_{[m]^{k+1};r}$ is the corresponding diamond in the Pl\"ucker frieze 
$\mathcal P_{(k,n)}$ given by (\ref{eq:matrix-def}) for $[m]^{k+1} = (m, m+1,\ldots, m + k)$ and $m \in [r+k,r-2]$.
Similarly to the first part of the proof (for $k \times k$ diamonds), one can show that conditions (c1) and (c2) are satisfied for  $[m]^{k+1}$ and for any order-inheriting subtuple thereof. 
So we can apply Proposition~\ref{P:determinant} to compute the determinants $b_{\underline{\tilde{m}_l};r}$ for $\underline{\tilde{m}_l} = (m, \ldots, \widehat{m+l}, \ldots, m+k)$, 
for $0 \leq l \leq k$, and Laplace expansion for the determinant of $A_{[m]^{k+1};r}$ by the last row yields 
\begin{eqnarray*}
\mathrm{det}(A_{[m]^{k+1};r}) 	&= & \sum_{l=1}^{k+1} (-1)^{k+1+l}p_{o( [r+k]^{k-1}, m+l-1 )} 
\cdot b_{\underline{\tilde{m}_l};r} \\
					&=&  \sum_{l=1}^{k+1} (-1)^{k+1+l}p_{o( [r+k]^{k-1}, m_l )} \cdot \prod_{j=0}^{k-2} p_{o([r+j]^k )} p_{o(  m_1, \ldots, \widehat{m_l}, \ldots, m_{k+1} )}  \\
					&=&  \prod_{j=0}^{k-2} p_{o([r+j]^k )} \cdot \underbrace{\sum_{l=1}^{k+1} (-1)^{k+1+l}p_{o( [r+k]^{k-1}, m_l )} p_{o(  m_1, \ldots, \widehat{m_l}, \ldots, m_{k+1} )})}_{(*)}
\end{eqnarray*}
Now since (*) is precisely the Pl{\"u}cker relation on $I = [r+k]^{k-1}$ and 
$J = [m_1,m_{k+1}]$, this term equals zero and so 
$ \mathrm{det}(A_{[m]^{k+1};r})=0$.  

It follows that the determinant of the $(k+1)\times (k+1)$ diamond $s(A_{[m]^{k+1};r})$ vanishes,
\[
	\det (s(A_{[m]^{k+1};r})) = s(\det(A_{[m]^{k+1};r})) = 0,
\]
and thus $s\mathcal P_{(k,n)}$ is tame. 
\end{proof}

\begin{remark}
	In \cite[Proposition 3.2.1]{mgost14}, it is shown that the variety of $\SL_k$-friezes of width $w=n-k-1$, say over $\mathbb{C}$, 
	can be embedded into $\mathrm{Gr}(k,n)$ as the subvariety of points which can be represented by matrices whose consecutive $k\times k$-minors all coincide (and are non-vanishing). 
	In the strategy from \cite[Proposition 3.2.1]{mgost14}, we could choose the $(k \times n)$-matrix $P$ with entries in $s\mathcal{A}(k,n)$ given by
	\[
	\footnotesize{
		\begin{bmatrix}
			 1  		& 	p_{o([n-k+2]^{k-1}2)}			& 		\ldots	& 	& 	p_{o([n-k+2]^{k-1}n-k)}	&  1 			& 0 	& \ldots 	& 0 \\
			0 & 1 	& p_{o([n-k+3]^{k-1}3)} 	& \ldots & p_{o([n-k+3]^{k-1}n-k)} 	& p_{o([k-1]^{k-1}n-k+1)}	& 1 	& \ldots 	& 0 \\
			\vdots & \vdots			& \ddots			&  & 				& 			& 	& \ddots 		&  \vdots \\
			0 & 0			& 	\ldots		& 1 & p_{o([1]^{k-1}k+1)} 	& p_{o([1]^{k-1}k+2)} & \ldots & p_{o([1]n-1)} 	& 1 
		\end{bmatrix}.
		}
	\]
Each consecutive $(k\times k)$-minor of the above matrix is $1$ (as follows e.g.\ from Proposition \ref{P:determinant}). We can complete this to a unique tame $SL_k$-frieze over $s\mathcal{A}(k,n)$. By Theorem \ref{thm:frieze} this must be precisely the specialised Pl\"ucker frieze of type $(k,n)$.

Let now $\mathcal{F}$ be any tame $SL_k$-frieze over $\mathbb{C}$. As explained in \cite{mgost14}, it is uniquely determined by a $k\times n$ slice of the form
\[
	M = \begin{bmatrix}
			1 & m_{12} 	& m_{13} & \ldots & m_{1,n-k} 	& 1 			& 0 	& \ldots 	& 0 \\
			0 & 1 			& m_{23} 	& \ldots & m_{2,n-k} 	& m_{2,n-k+1}	& 1 	& \ldots 	& 0 \\
			\vdots & \vdots			& \ddots			&  & 				& 			& 	& \ddots 		&  \vdots\\
			0 & 0			& 	\ldots		& 1  & 	m_{k,k+1}			& 		\ldots	& 	& 	m_{k,n-1}	&  1
			
		\end{bmatrix},
\]
representing a point in the cone over $\mathrm{Gr}(k,n)$ with respect to the Pl\"ucker embedding. The minors $M_{[r]^kl}$ of $M$, for $1 \leq r \leq k$ and $l \in [r+k+1, r-2]$ (with the interval taken cyclically modulo $n$), are precisely the entries of the matrix $M$, and we can view the matrix $M$ as a pointwise evaluation of the matrix $P$ at a point in the cone over the Grassmannian, and consequently, the frieze $\mathcal{F}$ as a pointwise evaluation of the specialised Pl\"ucker frieze $s\mathcal{P}_{(k,n)}$ at a point in the cone over the Grassmannian $\mathrm{Gr}(k,n)$. This might justify considering the specialised Pl\"ucker frieze $s\mathcal{A}(k,n)$ the ``universal'' $SL_k$-frieze of width $w = n-k-1$.

\end{remark}

%
\subsection{Integral tame $\SL_k$-friezes from Pl\"ucker friezes}\label{sec:integral}

\ 

As an application of Theorem~\ref{thm:frieze} we obtain that 
specialising a cluster to 1 yields an integral tame $\SL_k$-frieze as we will show now. \\

\begin{definition}\label{def:specialise}
	Let $\mathcal{A}$ be a cluster algebra of rank $m$, i.e., its clusters  have cardinality $m \in \ZZ_{>0}$. The specialization of a cluster 
$\underline{x} = (x_1, \ldots, x_m)$ in $\mathcal{A}$ to a tuple $(a_1, \ldots, a_m) \in \CC^m$ 
is the algebra homomorphism  
$\mathcal{A} \to \CC$ determined by sending 
$x_i$ to $a_i$ for all $1 \leq i \leq m$.
	If $(a_1, \ldots, a_m) = (1, \ldots, 1)$, we call this the specialization of $\underline{x}$ to $1$.
\end{definition}

\begin{remark}
Note that we consider clusters in a cluster algebra to be ordered tuples, rather than sets. 
\end{remark}

We observe here that specialising a cluster to a tuple $(a_1,\dots,a_m)$ determines 
values for all cluster variables, since the cluster algebra is generated by the cluster variables, each of which can be expressed as a Laurent polynomial in any given cluster. 
Since we 
are interested in integral $\SL_k$-friezes we consider specializations 
with respect to tuples in $\ZZ_{>0}^m$. 

\begin{remark}\label{R:image positive integers}
	The image of the cluster variables in $\mathcal{A}$ under a specialization of a cluster to a 
tuple in $(\ZZ_{>0})^m$ (or in $(\QQ_{>0})^m$) lies in $\QQ_{>0}$. This is due to the Laurent phenomenon (\cite{FZ-Laurent}) and positivity (see \cite{GHKK,LS2015,MSW}): 
Every non-initial cluster variable is a Laurent polynomial whose denominator is a monomial 
in the $x_i$ and whose 
numerator is a polynomial in the $x_i$ with positive coefficients . 
In particular, specialising a cluster to $1$ sends every cluster variable in 
$\mathcal{A}$ to a positive integer.
\end{remark}

\begin{remark}[Tameness]\label{R:image of a tame frieze}
Let $F$ be a tame $\SL_k$-frieze over an integral domain $R$, and let $\varphi \colon R \to S$ 
be a unitary ring homomorphism from $R$ to an integral domain $S$. Assume that the images 
of the entries of $F$ lie in $S'$ for some subring $S'$ of $S$. Then the grid $\varphi(F)$ we 
obtain by evaluating $\varphi$ entry-wise is a tame $\SL_k$-frieze over $S'$.
\end{remark}

\begin{cor}\label{cor:cluster-1-frieze}
Let $\underline{x}$ be a cluster in $\mathcal{A}(k,n)$ and let 
$\varphi_{\underline{x}} \colon \mathcal{A}(k,n) \to \CC$ be the specialization of 
$\underline{x}$ to $1$. Then $\varphi_{\underline{x}}(s\mathcal P_{(k,n)})$ is a tame 
integral $\SL_k$-frieze of width $w=n-k-1$. 
\end{cor}

\begin{proof}
By Theorem \ref{thm:frieze} and Remark~\ref{R:image positive integers}, 
$\varphi_{\underline{x}}(s\mathcal P_{(k,n)})$ is an integral $\SL_k$-frieze; 
the tameness follows from Remark~\ref{R:image of a tame frieze}, since $\varphi_{\underline{x}}$ 
is unitary. 
Its width follows from the definition of $\mathcal P_{(k,n)}$. 
\end{proof}

\begin{remark}
We see later (Lemma \ref{Lem:specialization-inj}) that 
for $k \leq 3$ and arbitrary $n$, as well as for $k=4$ and $n=6$ 
two different clusters $\underline{x}\ne\underline{x'}$ of a cluster algebra 
$\mathcal A$ produce different images $\varphi_{\underline{x}}(s\mathcal{P}_{(k,n)})$ and 
$\varphi_{\underline{x'}}(s\mathcal{P}_{(k,n)})$ if and only if $\underline{x}$ is not a 
permutation of $\underline{x'}$. 
\end{remark}

%
\section{Connection between the categories $\cC(k,n)$ and $\SL_k$-friezes}\label{sec:frieze-mesh}

In this section, we use indecomposable modules of 
the Grassmannian cluster categories $\cC(k,n)$ to form friezes 
of the same shape as the (special) Pl\"ucker friezes. 
We describe the Ext-hammocks for entries in these friezes 
and characterize the cases where such a frieze of modules 
(see Def.~\ref{def:frieze-of-modules}) 
gives 
rise to cluster-tilting objects. With that, we will then give a bijection between 
friezes on Auslander--Reiten quivers of Grassmannian cluster categories which we will call 
\emph{mesh friezes } (see Def.~\ref{def:mesh-frieze}), 
and integral tame $\SL_3$-friezes. 

%
\subsection{The Grassmannian cluster categories}\label{sec:grassm}

\ 

We recall the definition of and results about the Grassmannian cluster categories from~\cite{JKS16}. 
Note that an alternative method for constructing an additive categorification of Grassmannian cluster 
algebras with coefficients has been provided by Demonet and Luo \cite{DemonetLuo2016}; in the special 
case $k=2$ they applied Amiot's construction of the generalized cluster category \cite{AmiotCluster} 
to an ice quiver with potential coming from a triangulation of a polygon.
Let $Q(n)$ be the cyclic quiver with vertices $1,\dots,n$ and $2n$ 
arrows $x_i:i-1\to i$, $y_i:i\to i-1$. 
Let $B=B_{k,n}$ be the completion of the path algebra 
$\CC Q(n)/\langle xy-yx,x^k-y^{n-k}\rangle$, 
where $xy-yx$ stands for the $n$ relations $x_iy_i - y_{i+1}x_{i+1}$, $i=1,\dots,n$ 
and $x^k-y^{n-k}$ stands for the $n$ relations $x_{i+k}\dots x_{i+1} - y_{i+n-k+1}\dots y_i$ 
(reducing indices modulo $n$). 
Let $t=\sum x_iy_i$. The centre $Z=Z(B)$ is isomorphic to $\CC[[t]]$. 
Then we define the \emph{Grassmannian cluster category} of type $(k,n)$, 
denoted by $\cC(k,n)$, 
as the 
category of maximal Cohen Macaulay modules for $B$. In particular, the objects of $\cC(k,n)$ are  
(left) 
$B$-modules $M$, such that $M$ is free over $Z$. 
This category is a Frobenius category, and it is stably 2-CY. It is an additive categorification 
of the cluster algebra structure of 
$\mathcal A(k,n)$, as proved in \cite{JKS16}. Note that the stable category $\underline{\mathcal C(k,n)}$ is triangulated, which follows from \cite[Theorem 4.4.1]{BuchweitzMCM}, since $B$ is Iwanaga--Gorenstein, or from \cite[I, Theorem 2.6]{Happel88}, since $\cC(k,n)$ is a Frobenius category. 
Both $\cC(k,n)$ and its stable version $\underline{\cC(k,n)}$ will be called Grassmannian 
cluster categories.

We recall from \cite{JKS16} that the number of indecomposable non projective-injective 
summands in any cluster-tilting object of $\cC(k,n)$ 
is $(k-1)(n-k-1)$. This is called the {\em rank} of the cluster category $\cC(k,n)$. 
We say that $
\cC(k,n)$ is of {\em finite type}, if if has finitely many isomorphism classes of indecomposable modules. 
Note also that (for $k\le n/2$) 
the category $\cC(k,n)$ is of finite type if and only if 
either $k=2$ and $n$ arbitrary, or $k=3$ and $n\in \{6,7,8\}$. 
These categories are of Dynkin type $A_{n-3}$, $D_4$, $E_6$ and $E_8$ respectively. 
This means that one can find a cluster-tilting object in $\underline{\cC(k,n)}$ whose quiver 
is of the corresponding Dynkin type.

\begin{remark}[Indecomposable objects of $\cC(k,n)$]\label{rem:indec}  
The following results are  from~\cite{JKS16}. 
The rank of an object $M\in\cC(k,n)$ is defined to be the 
length of $M\otimes_ZK$ for $K$ the field of fractions of the centre $Z$ (\cite[Definition 3.5]{JKS16}). 
There is a bijection between the rank one indecomposable objects of $\cC(k,n)$ and $k$-subsets of $[1,n]$. We may 
thus write any rank one module as $M_I$ where $I$ is a $k$-subset of $[1,n]$. 
In particular, the indecomposable projective-injective objects 
are indexed by the $n$ consecutive $k$-subsets of $[1,n]$ 
(note here that we consider cyclic consecutive $k$-subsets modulo $n$.
The rank one indecomposables are thus in bijection with the cluster variables of 
$\mathcal A(k,n)$ which 
are Pl\"ucker coordinates, cf. Section~\ref{sec:pluecker-new}. 
We denote the bijective association $p_I\mapsto M_I$ by $\psi_{k,n}$. 
\end{remark} 

\begin{definition}  
\label{def:frieze-of-modules}
Using the bijection between rank one modules and Pl\"ucker variables, 
we can form a frieze of rank 1 indecomposable modules through the composition 
$\psi_{k,n}\circ\varphi_{(k,n)}$. For the remainder of this section, we will write $\mathcal M_{(k,n)}$ to 
denote the 
image of $\psi_{k,n}\circ\varphi_{(k,n)}$. 
We write $0$s for the images of the $p_I$ where 
$I$ has a repeated entry. \\
As an example, we take $(k,n)=(2,5)$. 
\[
\xymatrix@C=-8pt@R=-1pt{
&0 && 0 && 0 && 0 && 0 && 0 && 0 && 0\\
  && M_{23} && M_{34} && M_{45} && M_{15} && M_{12} && M_{23} && M_{34} \\
&M_{13} && M_{24} && M_{35} && M_{14} && M_{25} && M_{13} && M_{24} && \cdots  \\
\cdots && M_{14} && M_{25} && M_{13} && M_{24} && M_{35} && M_{14} && M_{25} \\
& M_{45}  && M_{15} && M_{12} && M_{23} && M_{34} && M_{45} && M_{15} && M_{12} \\
 && 0 && 0 && 0 && 0 && 0 && 0 && 0
}
\]

\end{definition}

\begin{remark}\label{rem:finite-types} 
In the finite types, the indecomposable modules are well-known. 
For $k=2$, they are exactly the rank 1 indecomposables. The Auslander-Reiten quivers 
of the categories $\cC(3,n)$, for $n=6,7,8$ are described in~\cite{JKS16}. We recall 
some of this information here. 
\\
(i) 
For $(3,6)$, there are 22 indecomposable objects. Among them, 20 are rank one modules. 
The additional two are rank two modules filtered by $M_{135}$ and $M_{246}$ (in both ways). Altogether, 
there are 6 projective-injective indecomposables and 16 other indecomposable objects. 
The Dynkin type of $\cC(3,6)$ is $D_4$. \\
(ii) 
For $(3,7)$, there are 14 rank 2 modules filtered by two rank 1 modules in addition to the 35 rank one modules. 
The category $\cC(3,7)$ has 7 projective-injectives objects and is a cluster category of type $E_6$. \\
(iii) In addition to the 
the 56 rank one modules, $\cC(3,8)$ has 56 rank two modules and 24 rank 3 modules, 
all filtered by rank 1 modules. It has 8 projective-injective objects. The Dynkin type of 
$\cC(3,8)$ is $E_8$. \\ 
The categories $\cC(3,9)$ and $\cC(4,8)$ are tame, these categories are known to be 
of tubular type, see~\cite{BBGE18}, where the 
non-homogenous tubes are described. \\ 
 
\end{remark}

%
\subsection{Description of Ext-hammocks}

\ 

We determine the shape of the {\em Ext-hammocks in $\mathcal M_{(k,n)}$},  
that is, given $I$ we describe the set of all $J$ appearing in the frieze $\mathcal M_{(k,n)}$
such that $\text{Ext}^1_{\mathcal{C}(k,n)}(M_I, M_J)$ is nonzero. 
We prove that the Ext-hammocks consist of two maximal rectangles determined by $I$.  
We use the result in \cite{JKS16}, where the authors determine a precise combinatorial condition for 
two rank 1 modules to have a nonzero extension.  For this section, we allow arbitrary $1<k<n-1$. 

Two $k$-subsets $I, J$ of $[1,n]$ are said to be {\it non-crossing} if there are no cyclically ordered $a, b, c, d$ with $a, c \in I \setminus  J$ and $b, d \in J \setminus  I$.  If such $a, b, c, d$ do occur, then $I$ and $J$ are {\it crossing}.  

\begin{prop}\cite[Proposition 5.6]{JKS16}\label{rank1-ext}
Let $I,J$ be $k$-subsets of $[1,n]$. \\ 
Then $\text{Ext}^1_{\mathcal{C}(k,n)} (M_I,M_J) = 0$ if and only if $I$ and $J$ are non-crossing.
\end{prop}

Next, we  define a maximal rectangle in the Pl{\"u}cker frieze starting (resp. ending) at $I$. 
Let $I=o([r]^{k-1},m)$ be a non-consecutive $k$-subset. Then we say that 
$J$ belongs to the {\em maximal rectangle starting at} $I$ if 
$J=o([s]^{k-1}, p)$ where $s\in \{r, r+1, \dots, m-k\}$ and $p\in \{m, m+1, \dots, r-2\}$.  
Similarly, we say that $J$ belongs to the {\it maximal rectangle ending at} 
$I$ whenever $s\in \{r, r-1, \dots m+2\}$ and $p\in \{ m, m-1, \dots, r+k\}$. 
Note that the collection of 
such elements indeed forms a maximal rectangle 
in $\mathcal M_{(k,n)}$. 

\begin{prop}\label{prop:ext-hammock}
Let $I = o([r]^{k-1}, m)$ be a non-consecutive $k$-subset and let $J$ be another element of $\mathcal{P}_{(k,n)}$. 
Then $\text{Ext}^1_{\mathcal{C}(k,n)}(M_I, M_J)$ is nonzero 
if and only if $J$ belongs to the maximal rectangle starting at $ o([r+1]^{k-1}, m+1)$ 
or the maximal rectangle ending at  $o([r-1]^{k-1}, m-1)$.
\end{prop}

\begin{proof}
By Proposition \ref{rank1-ext} it suffices to show that $I, J$ are crossing if and only if $J$ belongs to the maximal rectangle starting at $o( [r+1]^{k-1}, m+1)$ or the maximal rectangle ending at  $o( [r-1]^{k-1}, m-1)$.  Recall that 
$$I = o( [r]^{k-1}, m)= \{ r, r+1, \dots, r+k-2, m\}$$
and let 
$$J= o( [s]^{k-1}, p)=\{s, s+1, \dots, s+k-2, p\}$$
such that $I$ and $J$ cross.  In the figure below we depict elements in $I$ on a circle, and observe that $I$ and $J$ cross if and only if the following three conditions are satisfied.

\newlength\Colsep
\setlength\Colsep{10pt}

\begin{minipage}{\textwidth}
\begin{minipage}[c][4.5cm][c]{\dimexpr0.5\textwidth-0.5\Colsep\relax}
\hspace{.5cm}{\scalebox{.8}{
\begingroup%
  \makeatletter%
  \providecommand\color[2][]{%
    \errmessage{(Inkscape) Color is used for the text in Inkscape, but the package 'color.sty' is not loaded}%
    \renewcommand\color[2][]{}%
  }%
  \providecommand\transparent[1]{%
    \errmessage{(Inkscape) Transparency is used (non-zero) for the text in Inkscape, but the package 'transparent.sty' is not loaded}%
    \renewcommand\transparent[1]{}%
  }%
  \providecommand\rotatebox[2]{#2}%
  \ifx\svgwidth\undefined%
    \setlength{\unitlength}{146.513386bp}%
    \ifx\svgscale\undefined%
      \relax%
    \else%
      \setlength{\unitlength}{\unitlength * \real{\svgscale}}%
    \fi%
  \else%
    \setlength{\unitlength}{\svgwidth}%
  \fi%
  \global\let\svgwidth\undefined%
  \global\let\svgscale\undefined%
  \makeatother%
  \begin{picture}(1,0.86524655)%
    \put(0,0){\includegraphics[width=\unitlength]{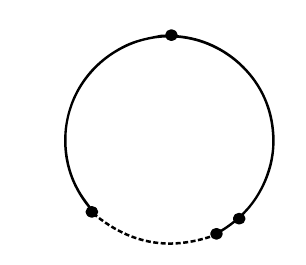}}%
    \put(0.53741428,0.7980172){\color[rgb]{0,0,0}\makebox(0,0)[lb]{\smash{$m$}}}%
    \put(0.85137877,0.13303653){\color[rgb]{0,0,0}\makebox(0,0)[lt]{\begin{minipage}{0.29641363\unitlength}\raggedright \end{minipage}}}%
    \put(0.81237716,0.12718618){\color[rgb]{0,0,0}\makebox(0,0)[lb]{\smash{$r$}}}%
    \put(0.73047321,0.10573521){\color[rgb]{0,0,0}\makebox(0,0)[lt]{\begin{minipage}{0.38026752\unitlength}\raggedright $r+1$\end{minipage}}}%
    \put(-0.00666085,0.16228783){\color[rgb]{0,0,0}\makebox(0,0)[lt]{\begin{minipage}{0.44657056\unitlength}\raggedright $r+k-2$\end{minipage}}}%
  \end{picture}%
\endgroup%
}} 
\end{minipage}\hfill
\begin{minipage}[c][4.8cm][c]{\dimexpr0.5\textwidth-0.5\Colsep\relax}
$J\cap (r+k-2,m)\not=\emptyset$

\vspace{.2cm}
$J\cap (m,r)\not=\emptyset$

\vspace{.2cm}
$m\not \in J$
\end{minipage}
\end{minipage}

Note that $p\not\in [r, r+k-2]$, because otherwise for the first two conditions above to be satisfied we must have $m\in [s]^{k-1}\subset J$.  This clearly contradicts the third condition.  Therefore, we have two cases $p\in (m,r)$ or $p\in (r+k-2, m)$.  

Suppose $p \in (m,r)$.  Then $I$ and $J$ cross if and only if $s+j \in (r+k-2, m)$ for some $j\in \{0, 1, \dots, k-2\}$, and also $s+j\not=m$ for any such $j$.  This implies that $s+k-2<m$.  On the other hand, we must also have that $r<s$.  All together, we obtain $r<s<m-k+2$.   Thus, we see from Figure \ref{ext-hammock} that $J$ lies in the maximal rectangle starting at 
$o( [r+1]^{k-1}, m+1)$.  

Similarly, suppose $p\in (r+k-2, m)$.  Then $I$ and $J$ cross if and only if $s\in (m,r)$.  From Figure \ref{ext-hammock} we see that this mean $J$ lies in the maximal rectangle ending at 
$o( [r-1]^{k-1}, m-1)$.  This shows the desired claim.

\end{proof}

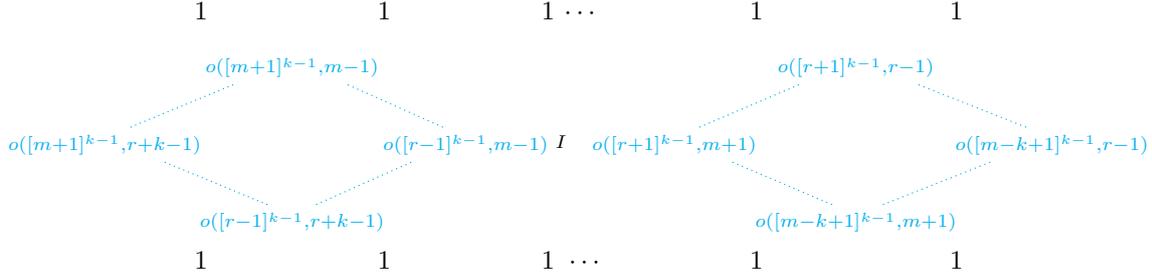
\begin{figure}
{\small
\vspace{0cm}
$$\xymatrix@C=-9pt@R=5pt{&1&&1&&1 && \cdots && 1 && 1 && \\
&& {\color{cyan} _{o( [m+1]^{k-1}, m-1)}}  \ar@{.}@[cyan][ddrr] && && && 
 && {\color{cyan} _{o( [r+1]^{k-1}, r-1)}} \ar@{.}@[cyan][ddrr]\\
&&&&&   &&  \\
{\color{cyan} _{o( [m+1]^{k-1}, r+k-1)}} \ar@{.}@[cyan][uurr]\ar@{.}@[cyan][ddrr]&& 
 && {\color{cyan} _{o( [r-1]^{k-1}, m-1)}} && \  _I 
 && {\color{cyan} _{o( [r+1]^{k-1}, m+1)}} \ar@{.}@[cyan][uurr] \ar@{.}@[cyan][ddrr]&& 
  &&  {\color{cyan} _{o( [m-k+1]^{k-1}, r-1)}} \\
&&&&&  && \\
&& {\color{cyan} _{o( [r-1]^{k-1}, r+k-1)}}  \ar@{.}@[cyan][uurr] && && && 
 && {\color{cyan} _{o( [m-k+1]^{k-1}, m+1)}} \ar@{.}@[cyan][uurr]\\
&1&&1&&1 &&\ \cdots && 1 && 1 && } $$}
\caption{Ext-hammock of $M_I$ where $I=o([r]^{k-1}, m)$}
\label{ext-hammock}
\end{figure}

For arbitrary $k$ and $n$, the category $\mathcal{C}(k,n)$ contains a cluster-tilting object 
consisting entirely of modules indexed by Pl\"ucker coordinates.  This follows from the results 
in \cite{Scott06}.  Next, we describe for which $k$ and $n$ there exists a cluster-tilting objected 
indexed by the almost consecutive Pl\"ucker coordinates.

\begin{prop}\label{prop:cluster-in-F} 
A Pl{\"u}cker frieze of type $(k,n)$ contains a collection $\mathcal{I}$ of almost consecutive 
$k$-element subsets such that $\oplus_{I\in\mathcal{I}} M_I$ is a cluster-tilting object of 
$\mathcal{C}(k,n)$ if and only if 
$k=2, 3$ with $n$ arbitrary, or $k=4$ and $n=6$. 
\end{prop}

\begin{proof}
In general, the cluster category $\mathcal{C}(k,n)$ contains a cluster-tilting object consisting of 
$(k-1)(n-k-1)$ indecomposable summands.  Therefore, we want to show that a 
Pl{\"u}cker frieze contains $(k-1)(n-k-1)$ distinct $k$-element subsets that are pairwise 
non-crossing if and only if $k=2,3$ or $k=4$ and $n=6$. 

If $k=2$ then every object in $\mathcal{C}(k,n)$ is of the form $M_I$ where $I$ is a 2 element subset of $[1, n]$.  In particular, every such $I$ lies in the associated Pl{\"u}cker frieze.  This shows the claim for $k=2$.  

If $k=3$ we want to find $2(n-4)$ pairwise non-crossing elements in the Pl{\"u}cker frieze.  Consider the following collection of 3-element subsets of $[1,n]$ consisting of two disjoint families of size $n-4$.
$$\{1,2,m\}_{m=4,5, \dots, n-1} \hspace{1cm} \text{and} \hspace{1cm} \{s, s+1, 1\}_{s=3, 4, \dots, n-2}$$
Each of these corresponds to a diagonal in the Pl{\"u}cker frieze, thus by Proposition 
\ref{prop:ext-hammock} we see that no two elements in the same family cross.  Moreover, $\{1,2,m\}$ does not cross $\{s,s+1,1\}$ by definition.   This shows that the collection of 3-element subsets above are pairwise non-crossing.

If $k\geq 4$ we want to show that the associated Pl{\"u}cker frieze does not contain $(k-1)(n-k-1)$ pairwise non-crossing elements unless $k=4$ and $n=6$.  
We make use of the following key observation here: the Ext-hammocks in a frieze of width $w$ have the same shape as the Ext-hammocks in the cluster 
category of type $A$ with rank $w$.  In particular, any region in the frieze that has the same shape as the fundamental region in the associated cluster 
category of type $A$, as depicted on the left in Figure \ref{fund_region}, contains at most $w$ pairwise non-crossing objects.  

First we consider the case $k=4$ separately.  We see that the Pl{\"u}cker frieze admits two copies of the fundamental region followed by two columns, because we have $w=n-k-1=n-5$ and the period of the frieze is $n$.  
See Figure \ref{fund_region} on the right.  Now suppose that such frieze admits a collection $\mathcal{I}$ of $3w$ pairwise non-crossing objects.  Because the frieze contains two such fundamental regions, each of which admits at most $w$ non-crossing object, we conclude that the two remaining columns contain at least $w$ element of $\mathcal{I}$.  However, a frieze is $n$-periodic, and the same argument as above implies that any two neighbouring columns contain at least $w$ elements of $\mathcal{I}$. Because there are $w+5$ columns we conclude that $\mathcal{I}$ contains at least 
$$\Big\lfloor \frac{w+5}{2}\Big\rfloor w$$ 
elements.   We also know that the size of $\mathcal{I}$ is $3w$ which implies that $\Big\lfloor \frac{w+5}{2}\Big\rfloor \leq 3$.  Therefore, there are two possibilities $w=1, k=4, n=6$ or $w=2, k=4, n=7$.  In the first case, we see from the diagram below that we can take $\mathcal{I}=\{1235, 1345, 1356\}$.

$$\xymatrix@C=2pt@R=2pt{1&& 1 && 1 && 1 && 1 && 1 && 1 && 1\\
&1235 && 2346 && 1345 && 2456 && 1356 && 1246 && 1235\\
1&& 1 && 1 && 1 && 1 && 1 && 1 && 1}$$

In the second case, one can check that it is not possible to find $\mathcal{I}$ of size 6.  This completes the proof for $k=4$.

Finally, suppose $k>4$ and let $\mathcal{I}$ be a collection of pairwise non-crossing elements in the Pl{\"u}cker frieze of size $(k-1)(n-k-1)=(k-1)w$.  By the same reasoning as in the previous case we think of subdividing the frieze into fundamental regions.  We begin by taking two copies of the frieze, so we have a rectangular array of elements of width $w$ and length $2n$. Next, we want to subdivide it into rectangles of width $w$ and length $w+3$, see Figure \ref{fund_region} on the right. Moreover, each such rectangular region contains at most $2w$ elements of $\mathcal{I}$.  But in the beginning we took two copies of the frieze, so dividing by two, we see that a Pl{\"u}cker frieze contains at most 
$$\Big\lceil \frac{2n}{w+3}\Big\rceil w$$
pairwise non-crossing elements.  To prove the claim it suffices to show that 
$$(k-1)w > \Big\lceil \frac{2n}{w+3}\Big\rceil w$$
Below, we will show that 
\begin{equation}\label{eq1}
k-1 > \frac{2n}{w+3}+1
\end{equation}
holds, which in turn implies the equation above.  Multiplying equation (\ref{eq1}) by $w+3$ and substituting $w+3=n-k+2$ we obtain 
\begin{equation}\label{eq2}
(k-4)n>(k-2)^2
\end{equation}
Since $k>4$ we can write
$$n>k-2 > \frac{(k-2)(k-2)}{k-4}$$
which in turn implies equations (\ref{eq2}) and $(\ref{eq1})$.  This completes the proof in the case $k>4$.

\begin{figure}
 {\scalebox{.8}{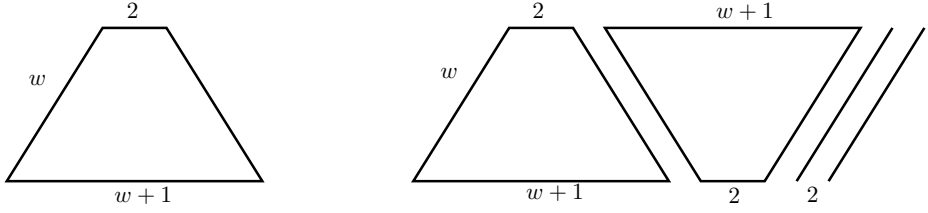}} 
 \caption{Fundamental region (left) and Pl{\"u}cker frieze where $k=4$ (right)}
 \label{fund_region}
\end{figure}
\end{proof}

In terms of cluster algebras we get the following:

\begin{cor}\label{cor:cluster-in-F-algebra}
A Pl{\"u}cker frieze of type $(k,n)$ contains a collection $\mathcal{I}$ of almost consecutive $k$-element 
subsets such that the collection 
$\{p_I\mid I\in \mathcal I\}$ is a cluster in $s\mathcal A(k,n)$ if and only if 
$k=2, 3$ with $n$ arbitrary or $k=4$ and $n=6$. 
\end{cor}

We can apply this to determine when the specialization $\varphi_{\underline{x}}$ of a cluster $\underline{x}$ to $(1, \ldots, 1)$ uniquely determines a specialized Pl\"ucker frieze.

\begin{lem} \label{Lem:specialization-inj}
Let $k \leq 3$ and $n$ arbitrary or let $k=4$ and $n=6$. Let $\underline{x}$ and $\underline{x}'$ be two clusters of Pl\"ucker variables. 
Then $\varphi_{\underline{x}}(s\mathcal{P}_{(k,n)})=\varphi_{\underline{x}'}(s\mathcal{P}_{(k,n)})$ 
if and only if 
$\underline{x}$ is a permutation of $\underline{x}'$. 
\end{lem}

\begin{proof}
The assumptions on $k$ and $n$ guarantee that the Pl\"ucker frieze $\mathcal{P}_{(k,n)}$ contains a cluster $\underline{a}$ of almost consecutive Pl\"ucker coordinates, see Prop.~\ref{prop:cluster-in-F} 
and Cor.~\ref{cor:cluster-in-F-algebra}. \\ 
If $\underline{x}$ is a permutation of $\underline{x}'$ then 
$\varphi_{\underline{x}} = \varphi_{\underline{x}'}$, which shows the reverse implication. 
On the other hand, assume
that there are two clusters $\underline{x}=(x_1, \ldots, x_N)$ and $\underline{x}'=(x'_1, \ldots, x'_N)$ of Pl\"ucker coordinates such that $\varphi_{\underline{x}}(s\mathcal{P}_{(k,n)})=\varphi_{\underline{x}'}(s\mathcal{P}_{(k,n)})$, where $N=(k-1)(n-k-1)+n$ is the cardinality of a cluster in $\mathcal A(k,n)$. 
Since $\underline{x}$ and $\underline{a}$ are clusters, there exist Laurent polynomials 
$Q_i^{\underline{x},\underline{a}}$ such that $a_i=Q_i^{\underline{x},\underline{a}}(x_1, \ldots, x_N)$. 
On the other hand, there also exist exchange Laurent polynomials $Q_i^{\underline{a},\underline{x}}$ such that $x_i=Q_i^{\underline{a},\underline{x}}(a_1, \ldots, a_N)$ and $Q^{\underline{a},\underline{x}}(Q^{\underline{x},\underline{a}}(\underline{x}))=\underline{x}$ (where we understand $Q^{\underline{x},\underline{a}}$ as the vector $(Q_1^{\underline{x},\underline{a}}, \ldots, Q_N^{\underline{x},\underline{a}})$). Similarly we can find $Q^{\underline{x}',\underline{x}}$ and $Q^{\underline{x},\underline{x}'}$. \\
By definition $\varphi_{\underline{x}}(\underline{x})=(1,\ldots, 1)$, and $\varphi_{\underline{x}'}(\underline{x}')=(1, \ldots, 1)$, and in particular, the entries of both evaluations agree on the cluster $\underline{a}$ in $s\mathcal{P}_{(k,n)}$. This means:
\begin{equation} \label{eq:evalx}
\varphi_{\underline{x}}(\underline{a})=\varphi_{\underline{x}}(Q^{\underline{x},\underline{a}}(\underline{x}))=Q^{\underline{x},\underline{a}}(1, \ldots, 1) \ . \end{equation}
On the other hand
\begin{equation} \label{eq:evalxstrich}
\varphi_{\underline{x}'}(\underline{a})=\varphi_{\underline{x}'}(Q^{\underline{x},\underline{a}}(Q^{\underline{x}',\underline{x}}(\underline{x}')))=Q^{\underline{x},\underline{a}}(Q^{\underline{x}',\underline{x}}(1, \ldots, 1)) \ . 
\end{equation}
Since $Q^{\underline{a},\underline{x}}(Q^{\underline{x},\underline{a}}(1, \ldots, 1))=(1,\ldots, 1)$, applying $Q^{\underline{a},\underline{x}}$ to \eqref{eq:evalx} and \eqref{eq:evalxstrich} yields 
$$(1, \ldots, 1)=Q^{\underline{a},\underline{x}}(Q^{\underline{x},\underline{a}}(Q^{\underline{x}',\underline{x}}(1, \ldots, 1)))=Q^{\underline{x}',\underline{x}}(1, \ldots, 1) \ . $$

Note that, while we have defined the cluster algebra $\mathcal A (k,n)$ as a $\mathbb{C}$-algebra, we have in fact 
\[
	\mathcal A (k,n) = \mathbb{C} \otimes_{\mathbb{C}} \mathcal A_\mathbb{Z} (k,n),
\]
where $\mathcal A_{\mathbb{Z}}(k,n)$ is a cluster algebra of geometric type, defined as the subring of $\mathbb{Q}(x_1, \ldots, x_N)$ generated by the cluster variables. Assume now as a contradiction that $\underline{x}$ is not a permutation of $\underline{x}'$. 
Then there is an $1 \leq i \leq N$ such that $x_i \notin \{x'_1, \ldots, x'_N\}$ and
\[
	x_i = Q_i^{\underline{x}',\underline{x}}(x'_1, \ldots, x'_N) = \frac{f(x'_1, \ldots, x'_N)}{{x'}_1^{d_1} \cdots {x'}_N^{d_N}},
\]
with $x'_j \nmid f(\underline{x}')$ 
for all $1 \leq j \leq N$. By positivity \cite{GHKK, LS2015} this is a Laurent polynomial over $\mathbb{Z}_{\geq 0}$. Therefore, $Q_i^{\underline{x}',\underline{x}}(x'_1, \ldots, x'_N) = 1$ implies that $f$ is the constant polynomial $f = 1$ and
\[
	x_i = \frac{1}{{x'}_1^{d_1}\cdots {x'}_N^{d_N}}.
\]
However, by positivity of $d$-vectors, which was shown for skew-symmetric cluster algebras by \cite[Thm.~1.2]{CaoLin}, we have $d_j \geq 0$ for all $1 \leq j \leq N$ and so $x_i$ is a unit in $\mathcal A_{\mathbb{Z}}(k,n)$; contradicting \cite[Theorem 1.3]{GLSFactorial}.
\end{proof}

%
\subsection{Mesh friezes} 

\ 

Assume $k\le n/2$. 
We recall that the Grassmannian cluster category $\cC(k,n)$ is of finite type if and only if either  
$k=2$, or $k=3$ and $n\in \{6,7,8\}$ 
(Section~\ref{sec:grassm}). 

\begin{definition}\label{def:mesh-frieze} 
Let $\cC(k,n)$ be of finite type. 
A {\em mesh frieze} $\mathcal{F}_{k,n}$ for the 
Grassmannian cluster category $\cC(k,n)$ is 
a collection of positive integers, one for each indecomposable object in $\cC(k,n)$ such 
the $\mathcal{F}_{k,n}(P)=1$ for every indecomposable projective-injective $P$ and such 
that all mesh relations on the AR-quiver of $\cC(k,n)$ evaluate to 1 thus, 
whenever we have an AR-sequence 
\begin{equation}\label{ar-triangle}
0\to A\to \bigoplus_{i=1}^I B_i \to C \to 0,
\end{equation}
where the $B_i$ are indecomposable, and $I \geq 1$, we get
\[
	\mathcal{F}_{k,n}(A)\mathcal{F}_{k,n}(C) = \prod_{i=1}^I \mathcal{F}_{k,n}(B_i) + 1.
\]
(Note that $\cC(k,n)$ has Auslander-Reiten sequences and an Auslander-Reiten quiver, cf.\ for example \cite[Remark~3.3 ]{JKS16}.)
\end{definition}

\begin{ex}\label{ex:mesh-sl2}
In case $k=2$, a mesh frieze $\mathcal{F}_{2,n}$ for $\cC(2,n)$ is an integral frieze pattern 
(or an integral $\SL_2$-frieze). In 
that case, the notion of mesh friezes and of integral $\SL_2$-friezes coincide. 
Such friezes are in bijection with triangulations of an $n$-gon and hence 
arise from specialising a cluster to $1$, see Example~\ref{ex:SL2}. 
\end{ex}

\begin{remark} \label{rem:mesh-ARS}
Note that \cite{ARS} introduce the notion of `friezes of Dynkin type'.  It coincides with certain
mesh friezes (on the stable version $\underline{\cC(k,n)}$) in types $A_n, D_4, E_6, E_8$.  
\end{remark}

\begin{remark}\label{rem:category-algebra} 
Recall that $\cC(k,n)$ is a categorification of the cluster algebra $\mathcal A(k,n)$. 
Furthermore, the mesh frieze has 1's at the projective-injective objects, i.e., the 
objects for consecutive $k$-subsets. 
In that sense, any mesh frieze $\mathcal{F}_{k,n}$ for $\cC(k,n)$ is also a mesh frieze associated 
to the cluster algebra $s\mathcal A(k,n)=\mathcal A(k,n)/J$ with coefficients specialized to 1, 
cf. Remark~\ref{rem:freeze-coeff}. 
The requirement that mesh relations evaluate to 1 translates to requiring that 
certain exchange relations in the cluster algebra hold. 
We will use both points of view. 
The approach to friezes via $\cC(k,n)$ is particularly useful when we want to check for 
the existence of cluster-tilting objects in the $\SL_k$-frieze. It will also 
be of importance in Section~\ref{sec:IY-reduction}
as we want to use Iyama--Yoshino reductions on these categories to analyse frieze patterns. 
\end{remark}

\begin{remark}\label{R:mesh frieze cluster relations}
Note that the entries in a mesh frieze satisfy the exchange relations of the associated cluster algebra $\mathcal{A}(k,n)$. That is, denoting by $x_N$ the cluster variable in 
$\mathcal{A}(k,n)$ associated to an indecomposable object $N$ in $\cC(k,n)$ the following holds: Whenever we have an exchange relation
\[
	x_Ax_C = \prod_{i=1}^{\tilde{m}} x_{\tilde{B}_i} + \prod_{j=1}^{\tilde{m}'}x_{\tilde{B}'_j}
\]
in $\mathcal A(k,n)$, then
\[
	\mathcal{F}(A) \mathcal{F}(B) = \prod_{i=1}^{\tilde{m}} \mathcal{F}({\tilde{B}_i}) + \prod_{j=1}^{\tilde{m}'}\mathcal{F}({\tilde{B}'_j}).
\]
Indeed, this follows from the fact that a mesh frieze is determined uniquely by the values of one cluster tilting object. Thus, we may take an acyclic cluster that forms a slice because $k=2$ and $n$ is arbitrary or $k=3$ and $n=6,7,8$, and then use sink/source mutations to determine all entries in the mesh frieze uniquely.  Now this exact same cluster also determines values on the mesh frieze uniquely using arbitrary mutations of the cluster algebra. Thus, these two approaches should yield the exact same mesh frieze, which implies that the entries in the mesh frieze satisfy all relations of the cluster algebra, not just the ones coming from sink/source mutations.
\end{remark}

%
\subsection{Mesh friezes for $\cC(k,n)$ yield $\SL_k$-friezes}\label{sec:mesh-SL}

\ 

In this section, we show the correspondence between integral tame $\SL_k$-friezes 
and mesh friezes in the finite type cases. For $k=2$, these two notions 
coincide (see Example~\ref{ex:mesh-sl2}). So for the rest 
of this section, we assume $k=3$ and $n\in \{6,7,8\}$. The width of such an $\SL_3$-frieze 
is $n-4$, see Remark~\ref{rem:width-period}. 

The following lemma shows that the entries in a tame $\SL_3$-frieze are determined by a cluster of $\mathcal{A}(k,n)$, which consists of consecutive Pl\"ucker coordinates.

\begin{lem}\label{lm:cluster-SL3-determined}
Let $F_{3,n}$ with  
$n \in \{6,7,8\}$ be a tame integral $\SL_3$-frieze. 
Then the entries of $F_{3,n}$ are uniquely determined by the values on the $2w$ entries $x_1=m_{00},x_2=m_{01}, \dots, x_w=m_{0,n-5}$ and 
$x_{w+1}=m_{2,n-3}, \dots, x_{2w}=m_{n-3,n-3}$.  \\
See below the picture for $(3,7)$. \\
\[   
\xymatrix@C=.08em@R=.06em{
   & 0 && 0 &&  0 && 0 && 0 && 0 && 0 \\
   0 && 0 &&  0 && 0 && 0 && 0 && 0 \\
 & 1 & &1 & & 1 &  & 1 &  & 1 &  & 1 & & 1 \\
 m_{22} & & m_{33} && {\color{red} x_6} && m_{55} && m_{66} &&  {\color{red} x_1} && m_{11} & \\ 
& m_{23} &&  {\color{red} x_5} && m_{45} && m_{56} && m_{60} &&  {\color{red} x_2} && m_{12}   \\
 m_{13} & &  {\color{red} x_4} && m_{35} && m_{46} && m_{50} && m_{61} &&  {\color{red} x_3} \\ 
  & 1 & &1 & & 1 &  & 1 &  & 1 &  & 1 & & 1 \\
     0 && 0 &&  0 && 0 && 0 && 0 && 0 \\
     & 0 && 0 &&  0 && 0 && 0 && 0 && 0 \\
}
\]

\end{lem}

\begin{proof}  We show the assertion for $F_{3,7}$. Here $w=3$. Assume that the values of the entries $(x_1, \ldots, x_6)$ are given. We first complete the frieze by repeating the entries along the diagonals (according to the (anti-)periodicity rule given in \cite{mgost14, Cuntz}), 
that is, the frieze becomes infinite also in the vertical directions. See below the picture for $F_{3,7}$: 
\[   
\xymatrix@C=.08em@R=.06em{
  & & \vdots && \vdots && \vdots && \vdots &&  \vdots && \vdots & \\ 
\dots & 1 & &1 & & 1 &  & 1 &  & 1 &  & 1 & & 1 & \dots \\
 m_{22} & & m_{33} && {\color{red} x_6} && m_{55} && m_{66} &&  {\color{red} x_1} && m_{11} & \\ 
& m_{23} &&  {\color{red} x_5} && m_{45} && m_{56} && m_{60} &&  {\color{red} x_2} && m_{12}   \\
 m_{13} & &  {\color{red} x_4} && m_{35} && m_{46} && m_{50} && m_{61} &&  {\color{red} x_3} \\ 
  & 1 & &1 & & 1 &  & 1 &  & 1 &  & 1 & & 1 \\
   0 && 0 &&  0 && 0 && 0 && 0 && 0 \\
   & 0 && 0 &&  0 && 0 && 0 && 0 && 0 \\
      1 & &1 & & 1 &  & 1 &  & 1 &  & 1 & & 1 \\
      & m_{66} &&  {\color{red} x_1} && m_{11} && m_{22} && m_{33} &&  {\color{red} x_6} && m_{55}   \\
       m_{56} & & m_{60} &&  {\color{red} x_2} && m_{12} && m_{23} &&  {\color{red} x_5} && m_{45} & \\ 
             & m_{50} && m_{61} &&  {\color{red} x_3} && m_{13} &&  {\color{red} x_4} && m_{35} && m_{46}   \\
\dots & & 1 && 1 && 1 && 1 && 1 && 1 &&  \dots \\
 & \vdots & &\vdots & & \vdots &  & \vdots &  & \vdots &  & \vdots & & \vdots &  \\
}
\]

 Then  we can successively compute the remaining entries $m_{ij}$ of the frieze from the $x_i$'s using the tameness condition, in particular the two rows of $0$s for calculating the $4\times4$ determinants. First compute $m_{35}$ from the determinant of the $4 \times 4$ matrix $A=\begin{pmatrix}  0 & 1 & x_4 & x_5 \\ 0 & 0 & 1 & m_{35} \\ 1 & 0 & 0 & 1 \\ x_1 & 1 & 0 &0 \end{pmatrix}$: since by tameness $\det(A)=0$, this yields $m_{35}=\frac{x_1 + x_5}{x_4}$. Moreover from the determinant of the matrix $\begin{pmatrix}  0 & 0 & 1 & m_{35} \\ 1 & 0 & 0 & 1  \\ x_1 &1 & 0 & 0 \\x_2&  m_{11} & 1 & 0 \end{pmatrix}$ we obtain $m_{11}=\frac{x_2 + m_{35}}{x_1}$. Similarly, we obtain $m_{61}=\frac{x_2 + x_6}{x_3}$ and $m_{33}=\frac{x_5 + m_{61}}{x_6}$. Next we obtain $m_{45}, m_{46}$ from the determinants of $\begin{pmatrix}  0 & 1 & m_{35} & m_{45} \\ 0 & 0 & 1 & m_{46} \\ 1 & 0 & 0 & 1 \\ m_{11} & 1 & 0 &0 \end{pmatrix}$ and $\begin{pmatrix}  x_4 & x_5 & x_6 \\ 1 & m_{35} & m_{45}  \\   0 & 1 & m_{46} \end{pmatrix}$: These two determinants yield the system of linear equations: 
\[ \underbrace{\begin{pmatrix} 1 & -m_{35} \\ -x_4 & (x_4 m_{35} -x_5) \end{pmatrix}}_{B}\begin{pmatrix} m_{45} \\ m_{46} \end{pmatrix} = \begin{pmatrix} -m_{11} \\ 1-x_6 \end{pmatrix} \ . \]
This system has a unique solution $(m_{45}, m_{46})$, since the determinant of the matrix $B$ equals $-x_5 \neq 0$. Similarly, we may also uniquely determine $m_{50}$ and $m_{60}$. The remaining entries may be calculated entirely by using determinants of the $3 \times 3$ diamonds.  
\\
The assertion for the cases $(3,6)$ and $(3,8)$ are proven by analogous tedious calculations. 
\end{proof}

\begin{prop}\label{prop:SL-mesh}
Let $w\in \{2,3,4\}$ and $n=w+4$. 
For every tame integral $\SL_3$-frieze of width $w$ there exists a  
mesh frieze on $\cC(3,n)$. 
\end{prop}

\begin{proof}
Let $F=F_{3,n}$ be a $\SL_3$-frieze of horizontal period $n=w+4$, for $w\in \{2,3,4\}$. 
Consider  
the Pl\"ucker frieze $\mathcal{P}(k,n)$ of type $(3,n)$. It has the same width and period as $F$, and we can overlay $\mathcal{P}(k,n)$ on top of $F$ such that every $p_I$ for an almost consecutive $3$-subset $I$ of $[1,n]$ comes to lie on top of an entry in $F$. (Note that a choice is involved here, the way to overlay $F$ by $\mathcal{P}(k,n)$ is not unique.) This assigns to each entry $p_I$ in the Pl\"ucker frieze $\mathcal{P}(k,n)$ an integer $f(p_I)$ lying underneath $p_I$.
By Proposition~\ref{prop:cluster-in-F}, the Pl\"ucker frieze $\mathcal{P}(k,n)$ contains a 
collection of entries $\{p_I \mid I \in \mathcal{I}\}$ such that $\bigoplus_{I \in \mathcal{I}} M_I$ is 
a cluster-tilting object of 
$\cC(3,n)$.
Its image under $\psi_{3,n}^{-1}$ 
is a cluster in $\mathcal A(3,n)$ 
(cf. Corollary~\ref{cor:cluster-in-F-algebra}). 
Specializing the indecomposable object $M_I$ for $I \in \mathcal{I}$ to $f(p_I)$ and calculating the remaining entries using the mesh relations yields a mesh frieze  
on the AR-quiver of $\mathcal C(3,n)$: 
By construction, the mesh relations are satisfied and 
as explained in 
Remark~\ref{R:image positive integers} all its entries are positive rationals. 
It remains to show that the entries are all integers. For this, we associate positive 
integers from $F$ to their positions in the corresponding Auslander-Reiten quiver/mesh 
frieze and use the mesh relations to see that all the other entries are sums and differences 
of products of these given positive integers. That proves the claim. 
We include pictures of the Auslander-Reiten quivers of these three categories along the way. 

(i) 
Let $w=2$ and $n=6$. 
The Auslander--Reiten quiver of $\cC(3,6)$ is shown in Figure~\ref{fig:AR36}. The indecomposables 
modules are indicated by black circles. Among them, the one corresponding 
to indecomposable rank 1 modules (or Pl\"ucker coordinates) 
appear with their label. The remaining two nodes correspond to rank 2 modules. 

The entries of $F_{3,6}$ yield positive integers in the 
$\tau$-orbits of the nodes $1$, $3$ and $4$ 
(see Figure~\ref{fig:Dynkin})
of the Auslander-Reiten quiver of $\cC(3,6)$, as this is 
where the objects indexed by almost consecutive $3$-subsets sit. 
To compute the entries of the $\tau$-orbit of node 2, we use the mesh relations. 
Here, they are $\SL_2$-relations. They involve the $\tau$-orbit of node 1 and the $1$'s above, 
if present. In particular, every entry in the $\tau$-orbit of node 2 is of the form 
$d=\frac{ab-1}{1}=ab-1$, for $a$ and $b$ in the $\tau$-orbit of 
node 1. So $d\in \ZZ$ and we have a mesh frieze. 
The uniqueness follows since the entries in $F_{3,6}$ are uniquely determined 
by the choice of the entries corresponding to a cluster. 

\begin{figure}
\[
\includegraphics[width=8cm]{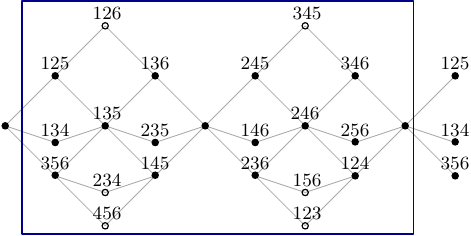}
\]
\caption{Auslander-Reiten quiver of $\cC(3,6)$. Projective-injectives are drawn as white circles.}\label{fig:AR36}
\end{figure}

(ii) 
The 
Auslander--Reiten quiver of $\cC(3,7)$ is in Figure~\ref{fig:AR37}. 
The unlabelled nodes correspond to rank 2 modules. 
Any integer $\SL_3$-frieze $F=F_{3,7}$ of width 3 yields positive integers in the 
$\tau$-orbits 
of the nodes $1$ (or $6$) and $2$, as the almost consecutive labels are exactly in 
these $\tau$-orbits. (Note that the $\tau$-orbits of node 1 and node 6 coincide). 
Using the frieze relation 
from Conway-Coxeter (see Equation (2) in \cite{BPT}) 
one can write all entries of the mesh frieze as subtractions of multiples of the entries in the 
$\tau$-orbit of node 1. So as in (i), all entries are integers and we have a mesh frieze.

(iii) 
Let $F=F_{3,8}$ be an integral $\SL_3$-frieze of width 4. Its entries yield 
positive integers in the $\tau$-orbits of the nodes $1$ and $8$ 
of the Auslander-Reiten quiver of $\cC(3,8)$ (see Figure~\ref{fig:AR38}). 
The unlabelled nodes are rank 2 and rank 3 modules. 
Observe that only 32 of the entries in the mesh frieze arise from the $\SL_3$-frieze in this case: 
The entries in the first and fourth row of $F$ form the $\tau$-orbit of 8, the entries of the second and 
third row the $\tau$-orbit of 1.
Using the frieze relation 
from Conway-Coxeter one can write all entries in the $\tau$-orbits of the nodes $3,4,5,6,7$ 
of the mesh frieze of $\cC(3,8)$ 
as subtractions of multiples of entries in the $\tau$-orbits above/below them 
(see Equation (2) in \cite{BPT}).
So again, 
all these entries are integers. To see that the entries in the $\tau$-orbit of node 2 are 
integers, use the Pl\"ucker relations. For example, to see that the entry for 
$I=\{2,5,7\}$ is an integer, we use the Pl\"ucker relations (\ref{eq:plucker-new}) 
for $(i_1,i_2)=(5,6)$ and $(j_0,\dots,j_3)=(2,4,5,7)$: 
$p_{562}p_{457} - p_{564}p_{257} + p_{565}p_{247} - p_{567}p_{245}=0$. This 
yields $p_{257}=p_{256}p_{457} - p_{567}p_{245}= p_{256}p_{457} - 1$ 
with $p_{256}\in\ZZ$, $p_{457}\in\ZZ$ (as they are Pl\"ucker coordinates of almost 
consecutive $3$-subsets). 
\end{proof}

\begin{figure}
\includegraphics[width=10cm]{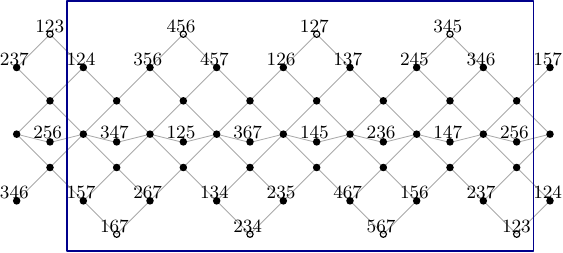}
\caption{Auslander-Reiten quiver of $\cC(3,7)$.}\label{fig:AR37}
\end{figure}

\begin{figure}
\[
\includegraphics[width=16cm]{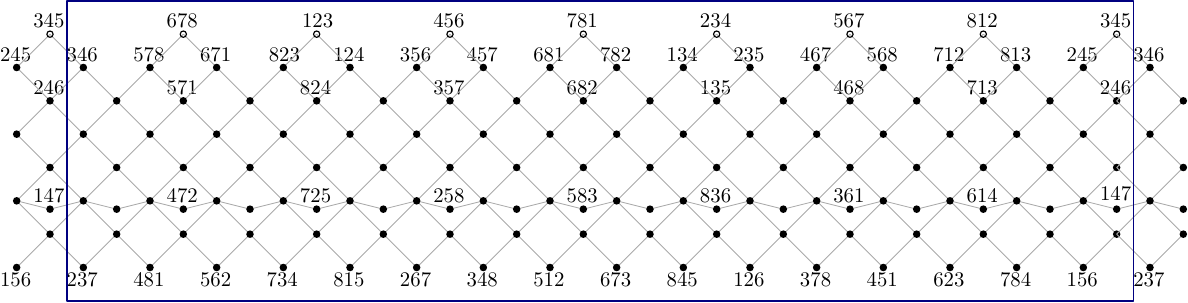}
\]
\caption{Auslander-Reiten quiver of $\cC(3,8)$.}\label{fig:AR38}
\end{figure}

\begin{cor}\label{cor:mesh-SLk}
Let $n\in \{6,7,8\}$. 
There is a bijection between mesh friezes for $\cC(3,n)$ and integral tame $\SL_3$-friezes of 
width $n-4$. 
\end{cor}

\begin{proof}
For $n=6$, this can be found in~\cite{Crossroads}. 
We use Proposition~\ref{prop:SL-mesh} to see that there is a surjective map from the 
set of mesh friezes for $\cC(3,n)$ to the set of integral tame $\SL_3$-friezes. 
This map is injective since by Lemma~\ref{lm:cluster-SL3-determined}, 
in these cases, the $\SL_3$-frieze contains a cluster which uniquely determines it. 
\end{proof}

%
\section{Mesh friezes via IY-reduction}\label{sec:IY-reduction}

In this section, we explain the effect of Iyama-Yoshino reduction on mesh friezes of cluster categories. 
We will consider cluster categories of not necessarily connected Dynkin type, i.e.\ the case where 
the Auslander-Reiten quiver consists of (possibly more than one) connected components, 
each of which is isomorphic to the Auslander-Reiten quiver of a cluster category 
$D^b(\mathrm{mod} H ) /\tau^{-1} \Sigma$, where $H$ is a hereditary algebra with quiver given by a 
disjoint union of Dynkin diagrams.

We say that such a cluster category is of finite type if it has only finitely many indecomposable objects. 
In this case, the {\em rank} is the number of summands in any cluster-tilting object. 

%

Note that 
when the stable category of $\mathcal{C}(k,n)$ is of finite type (i.e.\ if $k=2$ and $n$ is arbitrary, or if 
$k=3$ and $n \in \{6,7,8\}$), then its Auslander-Reiten quiver is isomorphic to the Auslander-Reiten quiver 
of the cluster category of the same type in the sense of \cite{BMRRT06}.

Recall that for any $\mathcal A(k,n)$, specialising a cluster to $1$ yields an 
integral tame $\SL_k$-frieze of width $w=n-k-1$ (Corollary~\ref{cor:cluster-1-frieze}). 
Mesh friezes for $\cC(2,n)$ all arise from specialising a cluster to 1. 
It is known that there exist mesh friezes on finite type cluster categories 
which do not arise from specialising a cluster, the first example is a mesh frieze in 
type $D_4$ containing only 2s and 3s, see Example~\ref{ex:D4} below. 
It has appeared in \cite[Appendix A]{BM09}, see also~\cite[Section 3]{FP16}, where 
all mesh friezes of type $D$ are characterised.  
For $\cC(3,6)$ there are thus 51 different mesh friezes: 50 arising from specialising a 
cluster to 1 and the example above.

%
\subsection{Mesh friezes and their reductions}

\ 

Let $M$ be a rigid indecomposable object of a triangulated 2-Calabi-Yau category $\cC$.
Let  $M^{\perp}:=\{Y\in \cC \mid \Hom_{\cC}(M,Y[1])=0\}=\{Y\in \cC\mid \Hom_{\cC}(Y,M[1])=0\}$ 
and let 
$\cC(M)=M^{\perp}/(M)$ where $(M)$ denotes the ideal of the category $M^\perp$ 
consisting of morphisms that factor through $\mathrm{add}(M)$. 
Then by \cite[Theorem~4.2, Theorem~4.7]{IY08}, 
$\mathcal C(M)$ is triangulated 2-Calabi-Yau. 
In particular, it has a Serre functor, and thus 
Auslander-Reiten triangles. 

If $\cC$ is a (stable) cluster category of finite type, every indecomposable object is rigid. 
Moreover, the reduction $\cC(M)$ is also of finite type. 
Let $\Gamma_{\cC}$ be the Dynkin diagram associated to $\cC$. See Figure~\ref{fig:Dynkin}. 
\begin{figure}[h]
\[
A_n: \ \includegraphics[width=3cm]{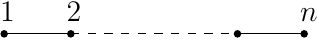} 
\hskip 1cm
D_n: \ \includegraphics[width=3cm]{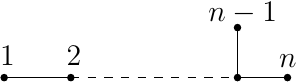} 
\]
\[
E_6: \ \includegraphics[height=1cm]{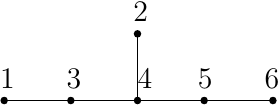}
\hskip 1cm
E_7: \ \includegraphics[height=1cm]{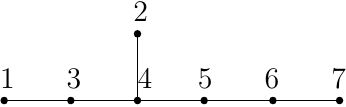}
\hskip 1cm
E_8: \ \includegraphics[height=1cm]{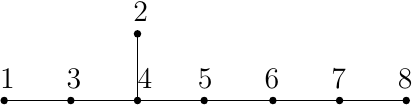}
\]
\caption{The Dynkin diagrams $\Gamma_{\cC}$ of the finite type cluster categories}\label{fig:Dynkin}
\end{figure}
Then  
the Auslander-Reiten 
quiver of $\cC$ has the shape of a quotient of the AR-quiver of the bounded derived category of this 
type, i.e.\ of 
$\Gamma_{\cC}\times \ZZ$ (cf.~\cite{Happel88}). 
For $M$ indecomposable, consider the indecomposable objects reached by 
sectional paths from $M$. The arrows of these paths form an oriented Dynkin diagram 
(see Figure~\ref{fig:Dynkin}) 
whose vertices are $M$ and the indecomposable objects on the sectional paths. 
This is a slice in the Auslander-Reiten quiver. 
The vertex of $M$ in the Dynkin diagram will be called the {\em node of $M$}. 

By considering the Auslander-Reiten quiver of $\cC(M)$, it is straightforward to prove the following statement 
(note that $\cC(M)$ may be a product of cluster categories of finite type): 

\begin{prop}\label{prop:reduce-type}
Let $\cC$ be a cluster category of type A,D or E. Then $\cC(M)$ is a cluster category. 
Its type is obtained by deleting the node of $M$ of the underlying Dynkin 
diagram. 
\end{prop}

\begin{notation}
Let $\mathcal F$ 
be a mesh frieze on a finite type cluster category $\cC$.  
Let $M$ be an indecomposable object in $\cC$. 
We write $\mathcal F(N)$ for the (integer) value of an indecomposable $N$ in $\mathcal F$. 
Then $\mathcal F|_{\cC(M)}$ denotes the collection of numbers of $\mathcal F$ on the AR-quiver of 
$\cC(M)$. 
\end{notation}

\begin{prop}\label{prop:restrict-frieze}
Let $\mathcal F$ be a mesh frieze on a finite type cluster category $\cC$, containing an entry 1. 
Let $M\in \ind \cC$ such that the frieze entry at $M$ is 1. Then 
$\mathcal F|_{\cC(M)}$ is a mesh frieze for $\cC(M)$. 
\end{prop}

\begin{proof}
Denote by $\langle 1 \rangle$ suspension in $\mathcal{C}(M)$ and let 

\begin{equation}\label{ar-triangle}
A\to B \to C \to A\langle 1\rangle 
\end{equation}
be an AR-triangle in $\mathcal{C}(M)$.  We need to show that the corresponding values of the mesh frieze satisfy the equation $\mathcal F(A)\mathcal F(C)=\mathcal F(B)+1$, where 
$\mathcal F(B)=\mathcal F(B_1)\cdots \mathcal F(B_t)$ whenever $B=\oplus_{i=1}^{t} B_i$.

Let $T$ be a cluster tilting object in $\mathcal{C}(M)$ consisting of all indecomposable objects $N$ that admit a sectional path starting at $A$ and ending at $N$.  Note that $T$ forms a slice in the AR-quiver of $\mathcal{C}(M)$ with $A$ being a source.  Moreover, $A,B\in \text{add}\, T$ while $C\not\in \text{add}\,T$, and the triangle (\ref{ar-triangle}) corresponds to a minimal left $\text{add}\,(T/A)$-approximation of $A$.  By construction, the minimal right $\text{add}\,(T/A)$-approximation of $A$ is zero.

\begin{equation}\label{ar-triangle2}
A\langle-1\rangle \to C \to 0 \to A 
\end{equation}

Next consider the object $\widetilde{T}=T\oplus M$ of $\mathcal{C}$.  Since, $T\in\mathcal{C}(M)$ it follows that $\text{Ext}^1_{\mathcal{C}}(T,M)=0$ 
Also, $\text{Ext}^1_{\mathcal{C}}(T,T)=0$, because $T$ is a cluster-tilting object in $\mathcal{C}(M)$.  Finally, $M$ is an indecomposable object of $\mathcal{C}$, which means that it is rigid.  This implies that $\widetilde{T}$ is rigid in $\mathcal{C}$.   Thus, $\widetilde{T}$ is a cluster-tilting object in $\mathcal{C}$ as it has the correct number of indecomposable summands.   

The AR-triangles (\ref{ar-triangle}) and (\ref{ar-triangle2}) in $\mathcal{C}(M)$ lift to the corresponding triangles in $\mathcal{C}$:

\begin{equation}\label{ar-triangle3}
A\to \widetilde{B} \to C \to A[1] \hspace{2cm} A[-1]\to C \to \widetilde{B}' \to A  
\end{equation}
where $\widetilde{B} \cong B \oplus M^i$ and $\widetilde{B}'\cong M^j$ for some non-negative integers $i,j$ 
and $[1]$ denotes suspension in $\mathcal{C}$. 
Here, $M^0$ denotes the zero module. 
Moreover, we see that the two triangles in (\ref{ar-triangle3}) are the left and right $\text{add}\,\widetilde{T}/A$-approximations of $A$ respectively.  

Let $\mathcal{A}$ be the cluster algebra associated to $\mathcal{C}$.  For an object $N\in \mathcal{C}$ let $x_N$ denote the associated (product of) cluster variable(s) in $\mathcal{A}$.  Because $\mathcal{C}$ is the categorification of $\mathcal{A}$ it follows that the two triangles (\ref{ar-triangle3}) give rise to the following exchange relation in the cluster algebra.

$$x_A x_C = x_{\widetilde{B}}+x_{\widetilde{B}'}$$ 

Note that the entries of the associated mesh frieze $\mathcal F$ also satisfy the relations of the cluster algebra 
$\mathcal{A}$. 
By Remark~\ref{R:mesh frieze cluster relations} we obtain the desired equation 

$$\mathcal F(A)\mathcal F(C)=\mathcal F(\widetilde{B})+\mathcal F(\widetilde{B}')=\mathcal F(B)\mathcal F(M^i)+\mathcal F(M^j)=\mathcal F(B)+1$$

because $\mathcal{F}(M^1) = \mathcal F(M)=1$ and $\mathcal F(M^0)=\mathcal F(0)=1$.
\end{proof}

\begin{remark}
We can use Proposition~\ref{prop:restrict-frieze} to argue that a mesh frieze for $\cC$ 
always has $\le m$ entries 1, if $m$ is the rank of $\cC$: 
Let $\mathcal F$ be a mesh frieze for $\cC$, a cluster category of rank $m$. Assume that 
$\mathcal F$ has $>m$ entries $1$. 
Take $M_0$ indecomposable such that its entry in the frieze is $1$. 
Let $\cC_1:=\cC(M_0)$. This has rank $m-1$. Let $\mathcal F_1=\mathcal F\mid_{\cC_1}$. 
Take $M_1\in \ind \cC_1$ such that the entry of $M_1$ in $\mathcal F_1$ is $1$. 
Let $\cC_2:=\cC_1(M_1)$. This is a cluster category of rank $m-2$ and 
$\mathcal F_2:=\mathcal F_1\mid_{\cC_2}$ has $\ge m-1$ entries equal to 1. Iterate until 
the resulting cluster category is a product of cluster categories of type A, whose 
rank is $m'$ (and $m'\ge 3$) 
and where the associated frieze $\mathcal F'$ has $\ge m'+1$ entries equal to $1$s. 
Contradiction.  
Here, we also might need to justify that passing from $\mathcal F$ to $\mathcal F_1$ we loose exactly one entry in the frieze that equals 1.  This holds, because suppose on the contrary that $\mathcal{F}(M_0)=\mathcal{F}(M_0')=1$, where $M_0'$ is some other indecomposable object of $\mathcal C$, and $M_0' \not \in M_0^\perp$.  This implies that $\text{Ext}^1_{\cC}(M_0, M_0')\not=0$.  In particular, there exist two triangles in $\cC$ 
$$M_0\to B \to M_0'\to M_0[1] \hspace{1cm} M_0'\to B' \to M_0 \to M_0'[1].$$
And by the same reasoning as in the proof of the above proposition, we obtain 
$$\mathcal{F}(M_0)\mathcal{F}(M_0') = \mathcal{F}(B)+\mathcal{F}(B').$$
The right hand side is clearly $\geq 2$, because $\mathcal{F}$ is an integral mesh frieze. This is a contradiction to $\mathcal{F}(M_0)=\mathcal{F}(M_0')=1$.
\end{remark}

%
\subsection{Counting friezes}

\ 

In this section, we study the number of mesh friezes for the Grassmannian 
cluster categories $\cC(3,n)$ in the finite types (types $D_4, E_6$ and $E_8$) 
and for the cluster category of type $E_7$. 
We recover the known number of mesh friezes for $\cC(3,7)$ and 
find all the known mesh friezes for $\cC(3,8)$ and for cluster categories 
of type $E_7$. 
Recall that specialising a cluster to 1 gives a unique mesh frieze on $\cC(k,n)$, 
cf. Corollary~\ref{cor:cluster-1-frieze}. 

\begin{definition}
A {\em unitary mesh frieze} is a mesh frieze which arises from 
specialising a cluster to 1. 
\end{definition}

Not every frieze pattern arises this way. Some mesh friezes have less 1's than 
the rank of the cluster category.  These are called {\em non-unitary friezes}. 
We will focus on these in what follows.

\begin{remark}\label{rem:count-known}
For $k=2$, the notion of mesh frieze coincides with 
the notion of $\SL_2$-frieze and the numbers of them are well-known, they are counted by the 
Catalan numbers. In particular, they all arise from specialising a cluster to 1. 
From now on, we thus concentrate on $k=3$ and $n\in \{6,7,8\}$. 
Recall first that the case $(3,6)$ corresponds 
to a cluster category of type $D_4$, the case $(3,7)$ to a cluster category of type $E_6$ and the 
case (3,8) to a cluster category of type $E_8$. 
The number of clusters 
in these types are 50, 833 and 25080 respectively, see e.g.~\cite[Table 3]{FZ2003}. 

It is known that the number of 
$\SL_3$-friezes of width 2 or, equivalently, the number of mesh friezes for $\cC(3,6)$
is 51 (this is proven via $2$-friezes in \cite[Prop.~6.2]{MGOT12},  and can be checked using 
a computer algebra system or counting the mesh friezes on a cluster category of 
type $D_4$, see Example~\ref{ex:D4} below). 
The number of $\SL_3$-friezes (equivalently, the number of mesh friezes for $\cC(3,7)$)
of width 3 is 868 by Theorem~\ref{theo::main}. 
Note that this number was already stated, citing a private communication by Cuntz, in 
\cite{Crossroads}. 
Appendix~\ref{sec:appendixB} now provides a proof for this statement

Fontaine and Plamondon found 
26952 mesh friezes for $\cC(3,8)$, see the lists on \cite{FP-lists}, see also \cite[Conjecture 2.1]{Cuntz} 
and  \cite[Section 4.4]{Crossroads}. 

We also recall that the number of clusters in a cluster category of type $E_7$ is 4160. 
Conjecturally, there are 4400 mesh friezes in type $E_7$ (\cite{FP-lists}). 
\end{remark}

\begin{cor}\label{cor:A-types}
Let $\mathcal F$ be a mesh frieze on a finite type cluster category $\cC$ and $M\in \ind\cC$ an object whose 
frieze entry is $1$. Then the following holds: \\
(1) If $\mathcal F':=\mathcal F|_{\cC(M)}$ is unitary, then $\mathcal F$ is unitary. \\
(2) 
If the Dynkin type of $\cC(M)$ is $A$ or a product of $A$-types, then $\mathcal F$ is unitary. 
\end{cor}

\begin{proof}
Let $n$ be the rank of $\cC$. 
By Proposition~\ref{prop:restrict-frieze}, $\mathcal F':=\mathcal F|_{\cC(M)}$ is a mesh frieze for $\cC(M)$. 
If $\cC(M)$ is unitary, then $\mathcal F'$ contains $n-1$ entries equal to 1. 
But then $\mathcal F$ contains $n$ 1s and is unitary. 
The second statement then follows, since in type A, all mesh friezes are unitary. 
\end{proof}

An immediate consequence of Corollary~\ref{cor:A-types} is the following. 
\begin{cor}\label{cor:unitary}
Let $\mathcal F$ be a mesh frieze on a cluster category of type $D$ or $E$. 
If $\mathcal F$ contains an entry 1 at the branch node or at a node which is an immediate neighbour of the 
branch node, then $\mathcal F$ is unitary. 
\end{cor}

By Corollary~\ref{cor:unitary}, non-unitary friezes cannot exist in ranks $\le 3$: In terms of ranks, 
the smallest possible cluster category 
with non-unitary mesh friezes is of type $D_4$. Indeed, we have one in this case, see Example~\ref{ex:D4}. 

\begin{ex}\label{ex:D4}
There exists a non-unitary mesh frieze in type $D_4$ (\cite[Appendix A]{BM09}. 
It is the only non-unitary mesh frieze in type $D_4$, see \cite[Section 3]{FP16}). 
\[
\xymatrix@C=.08em@R=.06em{
  & 2 && 2 && 2 && 2 &&   \\
 &  & 3 && 3 && 3 && 3  &&  \\
  \cdots &  && 2 && 2 && 2 &&  2 && \cdots\\
&  && 2 && 2 && 2 &&  2 && 
}
\]
\end{ex}

\subsubsection{Non-unitary mesh friezes in type $E_6$.}\label{ssec:E6}

Let $\cC$ be a cluster category of type $E_6$. 
Using the result of Cuntz and Plamondon (\ref{theo::main}), one can deduce that 
there are 35 non-unitary mesh friezes for $\cC$. 
One can extract these from the 
lists available on \cite{FP-lists}. One can check that each of them contains 
two 1s. 

Here, we explain how non-unitary mesh friezes arise in type $E_6$ and how any frieze with 
a 1 must contain at least two 1s. 
Observe that the AR-quiver of a cluster category of type $E_6$ consists of 7 slices of the 
Dynkin diagram $E_6$: 
\[
\xymatrix@C=.4em@R=.4em{
  & \bullet\ar@{-}[rd] && \bullet\ar@{-}[rd] && \bullet\ar@{-}[rd] && \bullet\ar@{-}[rd] && 
   \bullet\ar@{-}[rd] && \bullet\ar@{-}[rd] && \bullet\ar@{-}[rd] && \circ\ar@{.}[rd] \\ 
 && \bullet\ar@{-}[rd]\ar@{-}[ru] && \bullet\ar@{-}[rd]\ar@{-}[ru] && \bullet\ar@{-}[rd]\ar@{-}[ru] && 
  \bullet\ar@{-}[rd]\ar@{-}[ru] && \bullet\ar@{-}[rd]\ar@{-}[ru] && \bullet\ar@{-}[rd]\ar@{-}[ru] && \bullet\ar@{-}[rd]\ar@{-}[ru] && \circ \\
 & \bullet\ar@{-}[rd]\ar@{-}[ru]\ar@{-}[r] &\bullet\ar@{-}[r]  & \bullet\ar@{-}[rd]\ar@{-}[ru]\ar@{-}[r] & \bullet\ar@{-}[r] & 
   \bullet\ar@{-}[rd]\ar@{-}[ru]\ar@{-}[r] & \bullet\ar@{-}[r] & 
   \bullet\ar@{-}[rd]\ar@{-}[ru]\ar@{-}[r] & \bullet\ar@{-}[r] & 
   \bullet\ar@{-}[rd]\ar@{-}[ru]\ar@{-}[r] & \bullet\ar@{-}[r] & 
   \bullet\ar@{-}[rd]\ar@{-}[ru]\ar@{-}[r] & \bullet\ar@{-}[r] & 
   \bullet\ar@{-}[rd]\ar@{-}[ru]\ar@{-}[r] & \bullet\ar@{-}[r] &  \circ\ar@{-}[r]\ar@{.}[rd]\ar@{.}[ru] & \circ \\
 && \bullet\ar@{-}[rd]\ar@{-}[ru] && \bullet\ar@{-}[rd]\ar@{-}[ru] && \bullet\ar@{-}[rd]\ar@{-}[ru] && \bullet\ar@{-}[rd]\ar@{-}[ru] && 
 \bullet\ar@{-}[rd]\ar@{-}[ru] && \bullet\ar@{-}[rd]\ar@{-}[ru] && \bullet\ar@{-}[rd]\ar@{-}[ru]  && \circ\\
 & \bullet\ar@{-}[ru]&& \bullet\ar@{-}[ru] && \bullet\ar@{-}[ru] && \bullet\ar@{-}[ru] && \bullet\ar@{-}[ru] && \bullet\ar@{-}[ru] && 
  \bullet\ar@{-}[ru]  && \circ\ar@{.}[ru]
  }
\]
(Note that the last slice with circles involves a twist, the top node is identified with the bottom 
node in the first slice, etc.). 
 
By Corollary~\ref{cor:A-types}, if a non-unitary (mesh) frieze in type $E_6$ contains an entry 1, 
the node of this entry has to be 1 (or 6) (see Figure~\ref{fig:Dynkin} for the labelling of nodes). 
So let $\mathcal F$ be a non-unitary frieze for $\cC$ and 
$M$ an indecomposable corresponding to an entry 1 in the top row. 
Then $\cC(M)$ is a cluster category of type $D_5$. 
Furthermore, $\mathcal F|_{\cC(M)}$ is a non-unitary mesh frieze on $\cC(M)$ (since otherwise, $\mathcal F$ is unitary, by Corollary~\ref{cor:A-types}). 
There are exactly 5 non-unitary frieze on a cluster category of type $D_5$, the 
five $\tau$-translates of the following mesh frieze: 
\[
\xymatrix@C=.08em@R=.06em{ 
 2 && 3 && 1 && 3 && 2 \\ 
  & 5 && 2 && 2 && 5 && 3 \\
   && 3 && 3 && 3 && 7 && 7 && \\
  & && 2 && 2 && 2 && 4 && 2 \\
  & && 2 && 2 && 2 && 4 && 2 
}
\]
Since there are 7 positions for the first entry $1$ we picked (respectively for $M$), all in all 
there are exactly 35 different non-unitary friezes for $\cC$. 

\subsubsection{Non-unitary mesh friezes in type $E_7$.} \label{ssec:E7}

Let $\cC$ be a cluster category of type $E_7$. 
Conjecturally, there are 240 non-unitary mesh friezes for $\cC$. 
We now show how 240 non-unitary mesh friezes arise from non-unitary friezes 
for a cluster category of type $D_6$ or of type $D_4$. 
First recall that the AR-quiver 
of $\cC$ consists of 10 slices of a Dynkin diagram of type $E_7$. 
If $\mathcal F$ is a non-unitary for $\cC$, it cannot contain entries equal to 1 in the nodes 2,3,4 or 5 (see Figure~\ref{fig:Dynkin}) 
for the labels of the nodes in the Dynkin diagram).

\begin{remark}\label{rem:E7}
Let $\mathcal F$ be a non-unitary frieze for $\cC$. 
\begin{enumerate}[leftmargin=*,label=(i)]
\item Assume that $\mathcal F$ contains an entry 1 in the $\tau$-orbit 
of node 1, let $M$ be the corresponding indecomposable object. 
Then $\cC(M)$ is of type $D_6$. \\
(a) There are 2 mesh friezes of type $D_6$ without any 1's, see Example~\ref{ex:D6}. 
Hence we get 20 mesh friezes of type $E_7$ having an entry $1$ in the $\tau$-orbit of 1 and all 
other entries $\ge 2$. \\
(b) So assume that $\mathcal F'=\mathcal F_{\cC(M)}$ contains an additional entry 1, with 
corresponding module $N$. 
Then doing a further reduction $(\cC(M))(N)$ yields a cluster category of type $D_5$ or of type 
$A_1\times D_4$. In both cases, $\mathcal F'$ has two entries equal to 1. 
Analyzing the AR-quiver of 
$\cC(M)$ and the possible positions of these $1$s yields another 220 non-unitary mesh 
friezes, as one can check. 
\item The above mesh friezes cover 
are all known non-unitary mesh friezes in type $E_7$ having at least one entry 1, according to the 
lists of Fontaine and Plamondon. 
\end{enumerate}
\end{remark}

\begin{conj}\label{conj:E7}
There are no mesh friezes for $\cC$ of type $E_7$ where all entries are $\ge 2$. 
\end{conj}

\begin{ex}\label{ex:D6}
In type $D_6$ there are exactly two mesh friezes without any 1's, see~\cite[Section 3]{FP16}. One of them is 
here, the other one is the translate of it by $\tau$: 
\[
\xymatrix@C=.08em@R=.06em{
  & 2 && 2 && 2 && 2 && 2 && 2  \\
 && 3 && 3 && 3 && 3  && 3 && 3  \\
 &&& 4 && 4 && 4 && 4 && 4 && 4 \\
 \cdots &&&& 5 && 5 && 5 && 5 && 5 && 5 &&\cdots\\
 &&&&& 2 && 3 && 2 && 3 && 2 && 3 \\
&&&&& 3 && 2 && 3 &&  2 && 3 && 2
}
\]
\end{ex}

\subsubsection{Non-unitary mesh friezes in type $E_8$.} \label{ssec:E8}

Let $\cC$ be a cluster category of type $E_8$. Conjecturally, there are 
1872 non-unitary mesh friezes for $\cC$. 
Recall the labelling of the nodes in the Dynkin diagram 
from Figure~\ref{fig:Dynkin}. Non-unitary (mesh) friezes do not contain 1s in the $\tau$-orbits  
of the nodes 2,3,4 or 5, cf. Corollary~\ref{cor:unitary}. 

So in order to study non-unitary mesh friezes for $\cC$ we can start considering 
friezes containing an entry 1 in the $\tau$-orbit of node 6. 
There are 16 choices for this, as the Auslander-Reiten quiver of $\cC$ has 16 slices. 
Let $M$ be the corresponding indecomposable. 
Then $\cC(M)$ is a cluster category of type $D_5\times A_2$. 
There must be two more 1s in the frieze in the factor of type $A_2$ and one in the 
factor of type $D_5$. 
There are 5 choices of a cluster in type $A_2$ and there are five choices for an entry 
1 in the type $D_5$ factor (cf. Subsection~\ref{ssec:E6}. 
All in all there are $16 \times 5 \times 5=400$ possibilities for a non-unitary frieze 
with an entry 1 in the $\tau$-orbit of node 6. 

Continuing with similar arguments, one finds 1852 non-unitary friezes with four 1s and 
16 friezes with two 1s. 
There are no known mesh friezes for $\cC$ with only one entry $=1$, giving additional 
evidence for conjecture~\ref{conj:E7}: 

\begin{lem}
If Conjecture~\ref{conj:E7} is true, then every mesh frieze for $\cC$ of type $E_8$ contains 
0,2,4 or 8 entries equal to 1. 
\end{lem}

\begin{proof}
We only need to discuss non-unitary mesh friezes. \\
If there is an entry 1 in the $\tau$-orbit of node 8, 
the corresponding category $\cC(M)$ is of type $E_7$ and the claim follows by 
assumption (Conjecture~\ref{conj:E7}) and by Remark~\ref{rem:E7} (i) (a) and (b). \\
If there is an entry 1 in the $\tau$-orbit of node 1, $\cC(M)$ is a cluster category of type 
$D_7$. By the results in~\cite[Section 3]{FP16}, every non-unitary frieze of type $D_7$ 
contains either one or three entries $=1$. 
All the other cases arise from two entries equal to 1 and a non-unitary mesh frieze for 
a cluster category of type $E_6$, these are known to have two entries $=1$. 
\end{proof}

\begin{conj}\label{conj:E8}
If $\mathcal F$ is a mesh frieze for $\cC$ of type $E_8$ and if all entries of $\mathcal F$ are $\ge 2$, then 
$\mathcal F$ is one of the four friezes from Example~\ref{ex:E8}. 
\end{conj}

\begin{ex}\label{ex:E8}
In type $E_8$ there are 4 known mesh friezes without any 1's, appearing in the lists 
by Fontaine and Plamondon, 
available at \cite{FP-lists}. They are the 4 
translates of the following mesh frieze (a 4-periodic frieze, here,  8 slices are shown): 
\[
\xymatrix@C=.08em@R=.06em{
&   3 && 2 &&  2 && 3 && 3 && 2 && 2 && 3    \\ 
& & 5 && 3 && 5 && 8 && 5 && 3 && 5 && 8    \\ 
&  13 && 7 && 7 && 13 && 13 && 7 && 7 && 13  && \cdots  \\ 
\cdots &  & 18 && 16 && 18 && 21 && 18 && 16 && 18 && 21    \\ 
& 29 & 6 & 41 & 7 & 41 & 6 & 29 & 5 & 29 & 6 & 41 & 7 & 41 & 6 & 29 & 5   \\ 
 && 11 && 15 && 11 && 8 && 11 && 15 && 11 && 8  \\ 
&3 && 4 && 4 && 3 && 3 && 4&& 4 && 3 &&     
}
\]
These mesh friezes arise from specialising the following cluster to $(3,3,3,3,3,3,3,3)$. 
\[
\xymatrix@R=0.8em{
p_{235}\ar[r] & p_{236}\ar[dd] & p_{126}\ar[r]\ar[l] & p_{127}\ar[dd] \\
 && && \\
p_{356}\ar[uu] & p_{256}\ar[r]\ar[l] & p_{267}\ar[uu] & p_{167}\ar[l] 
}
\]
\end{ex}


\appendix

%
\section{Proof of Proposition~\ref{P:determinant}}  
\label{Appendix:ProofDet}

In the following, we are mainly interested in Pl\"ucker coordinates 
$p_{i_1,\dots, i_k}$ where a portion of the 
$i_l$'s is consecutive. For brevity we will write $[r]^l$ 
for the set $\{r, r+1, \ldots, r+l-1\}$. We reduce integers modulo $n$. 
So if $l=k-1$, such a tuple forms an almost consecutive $k$-subset. 
As before, we will use the map $o$ to indicate reordering by size, e.g. if $k=4$, $n=9$, 
$p_{o([8]^3),4}$ stands for the Pl\"ucker coordinate $p_{1894}$ and $p_{(o([8]^3,4)}$ 
for $p_{1489}$. 
We recall the proposition we are to prove here. 
 \begin{propos}[Proposition~\ref{P:determinant}]
 Let $r \in [1,n]$ and let $1 \leq s \leq k$. Let $\underline{m}=(m_1,\dots, m_s)$ 
 with $m_i \in [1,n]$ for all $i$, and assume that $\underline{m}$ satisfies conditions (c1) and (c2).
 Let $\bmr$ be the determinant of the matrix $\Amr$ 
 from (\ref{eq:matrix-def}). 

Then we have
 \begin{eqnarray} \label{E:result}
 	\bmr= \Big[ \prod_{l=0}^{s-2} p_{o([r+l]^k )} \Big] 
	\cdot p_{o( [r+s-1]^{k-s},  m_1,\ldots, m_s)}.
 \end{eqnarray}
\end{propos}

\begin{proof}[Proof of Proposition~\ref{P:determinant}]
If we have $m_i = m_j$ for some $1 \leq i < j \leq s$, then
\[
	\bmr =  0 
	= \Big[ \prod_{l=0}^{s-2} p_{o([r+l]^k )} \Big] \cdot p_{o( [r+s-1]^{k-s},  m_1,\ldots, m_s)}.
\]
as the last term is $0$. 
We can thus assume that $m_1, \ldots, m_s$ are mutually distinct.

	We prove the claim by induction over $s$. For $s=1$ the claim holds, since 
	$\bmr = \mathrm{det}(p_{o( [r]^{k-1}, m_1 )}) = p_{o( [r]^{k-1}, m_1 )}$. 
	Let now $2 \leq s \leq k$ and assume the statement is true for $1 \leq l < s$, 
	that is for all $b_{(\tilde{m}_1, \ldots, \tilde{m}_{l});s}$, where the $\tilde{m}_j$'s satisfy 
	conditions (c1) and 
	(c2).  
	We show that it also holds for $\bmr$.
	
	We have
	\begin{equation}\label{E:determinant}
	\begin{aligned}
\bmr &=\; \sum_{j=1}^s (-1)^{s+j} p_{o( [r+s-1]^{k-1}, m_j )} 
\cdot b_{(m_1, \ldots, \widehat{m_j}, \ldots, m_s);r} \\
&=\;   (-1)^s \cdot \Big[\prod_{l=0}^{s-3} p_{o([r+l]^k )} \Big] \cdot	\\
& \; \cdot \Big[ \sum_{j=1}^s (-1)^{j} p_{o([r+s-1]^{k-1}, m_{j})} \cdot p_{o([r+s-2]^{k-(s-1)},  m_1,\ldots, \widehat{m_{j}} \ldots, m_s)} \Big],					
	\end{aligned}
	\end{equation}
where the first equality is by Laplace expansion along the last row and the second equality is 
by induction assumption and pulling the term 
$X:= (-1)^s \cdot \prod_{l=0}^{s-3} p_{o([r+l]^k )}$ out of the sum. 
Notice that the conditions (c1) and (c2) pass down from the set 
$\{m_1, \ldots, m_s\}$ to any subset $\{m_1, \ldots, \widehat{m_j}, \ldots, m_s\}$, so we are 
justified in using the induction assumption on the minors of the form 
$b_{(m_1, \ldots, \widehat{m_j}, \ldots, m_s);r}$.

	We distinguish the following cases. 
	\newline
	{\bf Case 1:} Assume $s \neq k$ and there exists a $1 \leq l \leq s$ with 
	$m_l \in [r+s-2]^{k-(s-1)}$.
	\newline
	{\bf Case 1a:} Assume first that $m_l \neq r+s-2$. Then $m_l \in [r+s-1]^{k-s}$ and 
	the following expressions all vanish: 
	\begin{eqnarray*}
	p_{o([r+s-1]^{k-1}, m_l )}= p_{o([r+s-1]^{k-s}, m_1,\ldots, m_s)} = 0\\
	\text{and} \; \; p_{[r+s-2]^{k-(s-1)}, m_1,\ldots ,\widehat{m_{j}} \ldots , m_s} = 0 \; 
	\text{for all $j \neq l$} .
	\end{eqnarray*}
	Combining this with Equation (\ref{E:determinant}) it follows that
	\[
	\bmr = 0 = \prod_{l=0}^{s-2} p_{o([r+l]^k )} \cdot 
	p_{o( [r+s-1]^{k-s}, m_1,\ldots, m_s)}.
	\]
	This proves the claim in this case.
	\newline
	{\bf Case 1b:} Assume now that $m_l = r+s-2$ and $l<s$. If we had $m_{l+1} \in (r+k-2,m_l]$, then we would have 
		$r+k-2 \in (m_l,m_{l+1}) \subseteq (m_1,m_s)$ 
	contradicting condition (c2). Therefore, we must have 
	$m_{l+1} \in (m_l, r+k-2] = [r+s-1]^{k-s}$ 
	and we are in case 1a with $m_{l+1}$ taking on the role of $m_l$.
	\newline
	{\bf Case 1c:}  As a last subcase for case 1 assume that $l=s$ and $m_s = r+s-2$. Then all terms in (\ref{E:determinant}) with factors $p_{o( [r+s-2]^{k-(s-1)}, m_1, \ldots, \widehat{m_j}, \ldots, m_s )}$ with $j \neq s$ vanish and we obtain
	\[
	\begin{aligned}
		\bmr &=\;   (-1)^s \cdot \Big[\prod_{l=0}^{s-3} p_{o([r+l]^k )} \Big] \cdot 
								\\
								& \; \cdot (-1)^{s} p_{o([r+s-1]^{k-1}, m_{s})} \cdot p_{o([r+s-2]^{k-(s-1)}, m_1,\ldots, m_{s-1})},	\\
						    &=\;   \Big[ \prod_{l=0}^{s-2} p_{o([r+l]^k )} \Big] \cdot p_{o([r+s-1]^{k-s}, m_1,\ldots, m_{s})},				
	\end{aligned}
	\]
 where the second equality follows directly from the equality $m_s = r+s-2$. 
 This proves the claim in this case.	
	\newline
	{\bf Case 2:} Assume now $s = k$ and there exists a $1 \leq l \leq k$ such that $m_l = r+k-2$. This cannot happen if $l \neq k$: By condition (c2) we have $r+k-2 \notin [m_1,m_k)$, but by condition (c1), we have $m_l \in [m_1,m_k)$ for $1 \leq l < k$. Therefore, $m_l = r+k-2$ implies $l = k$, and thus in Case 2 we have $m_k = r+k-2$. Then the same argument as in case 1c, replacing $s$ by $k$, proves the claim.
	\newline
 {\bf Case 3:} We are now left with the case where $m_l \notin [r+s-2]^{k-(s-1)}$ for all $1 \leq l \leq s$.
	We are going to show that 
	\begin{eqnarray}\label{Eqn:to show}
	\begin{aligned}
	\sum_{j=1}^s (-1)^{j} & p_{o([r+s-1]^{k-1}, m_{j})} \cdot p_{o([r+s-2]^{k-(s-1)},  
	m_1,\ldots \widehat{m_{j}} \ldots, m_s)} = \\ 
	 & (-1)^s \cdot p_{o( [r+s-2]^k )} \cdot p_{o( [r+s-1]^{k-s}, m_1, \ldots, m_s )},
	\end{aligned}
	\end{eqnarray}
	which, when substituting into (\ref{E:determinant}) yields the desired equality (\ref{E:result}).
	We use the Pl\"ucker relations on $I =[r+s-1]^{k-1}$ of cardinality $|I| = k-1$ and $J = [r+s-2]^{k-(s-1)} \cup  \{m_1, m_2, \ldots, m_s\}$ of cardinality $|J| = k+1$.
	\newline
	{\bf Case 3a:} Assume that $1 = r+a \in [r+s-1]^{k-1} \cap [r+s-2]^{k-(s-1)}$. Then $I$ is ordered as follows:
	\[
		1 = r+a < r+a+1 < \ldots < r+s+k-3 < r+s-1 < \ldots < r+a-1
	\]
	and $J$ is ordered as follows:
	\[
	\begin{aligned}
		1 = r+a &< r+a+1 < \ldots < r+k-2 < m_1 < m_2 < \ldots < m_s \\ & < r+s-2 < r+s-1 < \ldots < r+a-1 = n.
	\end{aligned}
	\]
	Notice that we must have $m_1 < m_l$ for all $l \neq 1$. Otherwise, if there exists a $1 < l \leq s$ with $m_l < m_1$, then if we have $m_l \in [1, r+k-2] \subseteq [r+s-2, r+k-2]$ this contradicts the assumptions of case 3, and if we have $m_l \in (r+k-2, n]$, then it follows from $m_l < m_1$ that $r+k-2 \in [m_1, m_l) \subseteq [m_1,m_s)$, which contradicts condition (c2).
	The Pl\"ucker relation on $I$ and $J$ (cf.\ Equation (\ref{eq:plucker-new})) thus reads as:
	\[
	\begin{aligned}
	0 =	& \sum_{l=0}^{k-a-2} (-1)^l \cdot p_{o( [r+s-1]^{k-1} ), r+a+l} \cdot p_{o( J \setminus \{r+a+l\})} \\
		& + \sum_{l=1}^{s} (-1)^{k-a-2+l} \cdot p_{o( [r+s-1]^{k-1} ), m_l} \cdot p_{o( J \setminus m_l )} \\
		& + \sum_{l=0}^{a-s+1} (-1)^{k-a+s+l-1} \cdot p_{o( [r+s-1]^{k-1} ), r+s-2+l} \cdot p_{o( J \setminus \{r+s-2+l\})}.
	\end{aligned}
	\]
Notice that $p_{o([r+s-1]^{k-1}), r+j} = 0$ for all $j \in [s-1,k-2]$ 
and thus the whole first sum among those three vanishes, and all the summands, 
except for the first one 
($l=0$) in the last sum vanish as well. The Pl\"ucker relation thus reduces to
	\[
	\begin{aligned}
	0 =&	\sum_{l=1}^{s} (-1)^{k-a-2+l} \cdot p_{o( [r+s-1]^{k-1}), m_l} \cdot 
	p_{o( J \setminus m_l )}\\
	& +(-1)^{k-a+s-1} \cdot p_{o( [r+s-1]^{k-1}), r+s-2} \cdot 
	p_{o( [r+s-1]^{k-s}, m_1, \ldots, m_s)}.
	\end{aligned}
	\]
	We need to get the remaining $k$-tuples in linear order. Observe that
	\[
	\begin{aligned}
	p_{o([r+s-1]^{k-1}, m_l )} =& p_{r+a, \ldots, r+s+k-3, m_l, r+s-1, \ldots, r+a-1} \\ 
	=& \sgn(\pi) p_{o( [r+s-1]^{k-1}), m_l} 
	\end{aligned}
	\]
	and
\[
\begin{aligned}
p_{o([r+s-2]^k)} =& (r+a,\ldots, r+s+k-3, r+s-2, r+s-1, \ldots, r+a-1) \\ 
 =& \sgn(\pi) p_{o( [r+s-1]^{k-1}), r+s-2}
	\end{aligned}
	\]
where $\pi = (s+k-a-1, s+k-a, \ldots, k-1, k) \in S_k$.
Substituting this into the Pl\"ucker relation on $I$ and $J$ 
(and multiplying by $\mathrm{sign}(\pi) \cdot (-1)^{k-a}$) yields
	\[
	\begin{aligned}
0 =& \sum_{l=1}^{s} (-1)^{l} \cdot p_{o([r+s-1]^{k-1}, m_l)} \cdot p_{o(J \setminus m_l )}\\
&- (-1)^s \cdot p_{o([r+s-2]^{k} )} \cdot p_{o( [r+s-1]^{k-s}, m_1, \ldots, m_s)},
\end{aligned}
\]
which, observing that 
$p_{o( J \setminus m_l )} = p_{o([r+s-2]^{k-(s-1)}, m_1, \ldots, \widehat{m_l}, \ldots, m_s )}$ 
and pulling the second term to the left hand side yields the desired 
Equation (\ref{Eqn:to show}).
\newline
{\bf Case 3b:} For the final case, assume that $1 \notin [r+s-1]^{k-1} \cap [r+s-2]^{k-(s-1)}$. 
We first observe what the linear order on $J$ looks like. Notice that we have 
$r+s-2 < r+s-1 < \ldots < r+k-2$. Since $m_j \notin [r+s-2, r+k-2]$, the set $J$ is ordered in 
one of the following three ways (\ref{order 1}), (\ref{order 2}), (\ref{order 3}):
\begin{eqnarray}
\quad \quad m_{b} < m_{b+1} <  \ldots < m_s < r+s-2 < \ldots < r+k-2 <  m_1 < m_2 < 
\ldots < m_{b-1}, \label{order 1}
\end{eqnarray}
for some $b \in \{2, \ldots, s\}$,
\begin{eqnarray}
&m_{1} < m_{2} <  \ldots < m_s < r+s-2 < \ldots < r+k-2, \; \label{order 2}\text{or} \\ 
&r+s-2 < \ldots < r+k-2 < m_1 < \ldots < m_s. \label{order 3}
\end{eqnarray}
If it is ordered as in (\ref{order 2}), we set $b = 1$, and if it is ordered as in (\ref{order 3}), 
we set $b=s+1$. Then the Pl\"ucker relation on $I$ and $J$ reads as follows:
\[
\begin{aligned}
0 =& \sum_{l=b}^{s} (-1)^{l-b}p_{o(\langle[r+s-1]^{k-1}), m_{l}} \cdot p_{o(J \setminus m_{l})} \\
& + \sum_{l=0}^{k-s}(-1)^{s-b+1+l} \cdot p_{o([r+s-1]^{k-1}), r+s-2+l} 
\cdot p_{o(J \setminus \{r+s-2+l\})} \\
& + \sum_{l=1}^{b-1} (-1)^{k-b+1+l} \cdot p_{o([r+s-1]^{k-1}), m_{l}} \cdot p_{o( J \setminus m_{l})}.
	\end{aligned}
	\]
Notice again that $p_{o([r+s-1]^{k-1}), r+j} = 0$ for all $j \in [s-1, k-2]$, and therefore the corresponding terms in the Pl\"ucker relation vanish. Again, we need to get the remaining $k$-tuples in linear order. Observe that
	\[
	p_{o([r+s-1]^{k-1}), m_{l}} = 
	\begin{cases}
	p_{o( [r+s-1]^{k-1}, m_l)} & \text{if} \; 1 \leq l \leq b-1 \\
	\sgn(\sigma) p_{o( [r+s-1]^{k-1}, m_l)} & \text{if} \; b \leq l \leq s
	\end{cases}
	\]
and that $p_{o([r+s-1]^{k-1}) , r+s-2}= \sgn(\sigma) p_{o( [r+s-2]^k)}$, where 
$\sigma = (1, 2, \ldots, k) \in S_k$ with $\sgn(\sigma) = (-1)^{k-1}$. 
Therefore, if $J$ is ordered as in (\ref{order 2}) or (\ref{order 3}) the Pl\"ucker relation on $I$ and 
$J$ can be rewritten as
\[
\begin{aligned}
 0 =&  \sum_{l=b}^{s} (-1)^{l-b} \cdot (-1)^{k-1} p_{o([r+s-1]^{k-1}, m_{l} )} \cdot p_{o( J \setminus m_{l} )} \\
& + (-1)^{s-b+1} \cdot (-1)^{k-1} p_{o([r+s-2]^{k} )} \cdot p_{o( J \setminus \{r+s-2\})} \\
& + \sum_{l=1}^{b-1} (-1)^{k-b+1+l} \cdot p_{o([r+s-1]^{k-1}, m_{l})} \cdot p_{o( J \setminus m_{l})},
\end{aligned}
\]
Juggling the signs in a straight forward way, and pulling together the first and last sum 
we obtain
\[
\begin{aligned}
0 =& (-1)^{s}  \cdot (-1)^{k-b} \cdot p_{o([r+s-2]^{k} )} \cdot p_{o(  J \setminus \{r+s-2\})} \\		  
 & + (-1)^{k-b+1} \sum_{l=1}^{s} (-1)^{l} \cdot p_{o([r+s-1]^{k-1} , m_{l} )} 
\cdot p_{o( J \setminus m_{l})}.
\end{aligned}
\]
Pulling the second term on the left hand side and multiplying by $(-1)^{k-b}$ yields as desired Equality (\ref{Eqn:to show}),
observing that  $p_{o( J \setminus m_l )} = p_{o([r+s-2]^{k-(s-1)}, m_1, \ldots, \widehat{m_l}, \ldots, m_s )}$. 	
This concludes the proof.
\end{proof}

%
\section{Counting friezes in type $E_6$}\label{sec:appendixB}

\addtocontents{toc}{\protect\setcounter{tocdepth}{0}}

\noindent
\begin{center}
\begin{footnotesize}
{\sc Michael Cuntz$^{\#}$, Pierre-Guy Plamondon$^{\dagger}$} 
\end{footnotesize}\end{center}

Friezes (also called ``friezes of type~$A$'' and~``$\SL_2$-friezes'') were introduced by H.~Coxeter in \cite{Coxeter} and later studied by H.~Coxeter and J.~Conway in \cite{CoCo1,CoCo2}.
It was observed by P.~Caldero that the theory of cluster algebras of S.~Fomin and A.~Zelevinsky \cite{FZ2002} 
allows for a far-reaching generalization of the original notion of frieze;
this generalization was first studied in \cite{ARS}. 
Recall that in \cite{ARS}, 
a frieze (of Dynkin type) is 
a mesh frieze on an associated Auslander-Reiten quiver, 
cf. Remark~\ref{rem:mesh-ARS}.
Since then, many generalizations and variations on the notion of friezes have been introduced, as can be seen in the survey paper~\cite{Crossroads}.

It is known that, for a given non-Dynkin type, there are infinitely many  
friezes of that type.
However, it follows from \cite{CoCo1,CoCo2} that friezes of Dynkin type~$A$ come in a finite number 
(given by Catalan numbers),
and it was proved in~\cite{FP16} that friezes of Dynkin type~$B,D$ and~$G$ also come in a finite (explicit) 
number (the result for type~$D_4$ was also proved in~\cite{MGOT12}).
It was conjectured in \cite{Propp, MGOT12, FP16} that the number of 
friezes in type~$E_6$ is~$868$;
in~\cite{FP16} the precise number of friezes in any Dynkin type is also conjectured.

In this Appendix, we settle the case of type~$E_6$, and obtain the result for type~$F_4$ as a corollary.

\begin{theorem}\label{theo::main}
 The number of friezes of type~$E_6$ is exactly $868$.
\end{theorem}

Since Dynkin type~$F_4$ is a folding of type~$E_6$, it follows from the work of~\cite{FP16} that we then have:

\begin{corollary}
 The number of friezes of type~$F_4$ is exactly~$112$.
\end{corollary}

Our proof relies on a reduction to~$2$-friezes (whose definition we recall below);
our strategy is to show that the entries in a~$2$-frieze of height~$3$ are bounded.

We have attempted to apply the methods used in this Appendix to types~$E_7$ and~$E_8$, without success.

\section*{B.1. $2$-friezes}
We shall not be using the definition of a frieze from~\cite{ARS}, 
but rather a slightly different notion, that of a~$2$-frieze as defined in~\cite{MGOT12}.

\begin{definition}
 A \emph{$2$-frieze of height~$h$} is an array of positive integers~$(a_{i,j})$, where
 \begin{itemize}
  \item $i\in\ZZ$ and~$j\in \{0, 1, \ldots, h+1\}$;
  \item for all~$i\in \ZZ$, we have that $a_{i,0} = a_{i,h+1} = 1$;
  \item for all~$i\in \ZZ$ and all~$j\in \{1, \ldots, h\}$, we have that~$a_{i,j} = a_{i-1, j}a_{i+1, j} - a_{i,j-1}a_{i,j+1}$.
 \end{itemize}
\end{definition}

The reason we are interested in 2-friezes is the following result, which is a consequence of 
\cite[Proposition 5.4]{MGOT12} and of the fact that the quiver depicted in \cite[(2.3)]{MGOT12} 
is mutation-equivalent to a quiver of type $E_6$.

\begin{theorem}[\cite{MGOT12}]
Entries in a 2-frieze of height 3 are entries in a frieze of type $E_6$ 
which determine the frieze.  This defines an injection from the set of friezes of type 
$E_6$ to the set of 2-friezes of height 3.
\end{theorem}

Our strategy to prove Theorem \ref{theo::main} is thus to prove that the number of~$2$-friezes of height~$3$ is finite.
We will do this by showing that there is a bound on the possible values appearing in a~$2$-frieze of height~$3$. 

\section*{B.2. Two choices of initial variables for~$2$-friezes of height~$3$}
If we fix the entries of the first two rows of a~$2$-frieze of height $3$ to be~$s,t,u,v,w$ and~$x$, then we get the following expressions for all its entries:
      \begin{center}
        \begin{tikzpicture}
          \draw (0,1) node{$1$} (2,1) node{$s$} (4,1) node{$t$} (6,1) node{$u$} (8,1) node{$1$} ;
          \draw (0,0) node{$1$} (2,0) node{$v$} (4,0) node{$w$} (6,0) node{$x$} (8,0) node{$1$} ;
          \draw (0,-1) node{$1$} (2,-1) node{$\frac{v+w}{s}$} (4,-1) node{$\frac{vx+w}{t}$} (6,-1) node{$\frac{w+x}{u}$} (8,-1) node{$1$} ;
          \draw (0,-2) node{$1$} (2,-2) node{$\frac{svx+sw+tv+tw}{stv}$} (4,-2) node{$X$} (6,-2) node{$\frac{tw+tx+uvx+uw}{tux}$} (8,-2) node{$1$} ;
          \draw (0,-3) node{$1$} (2,-3) node{$\frac{suvx+suw+tuw+tvw+tvx}{tuvw}$} (4,-3) node{$Y$} (6,-3) node{$\frac{stw+suvx+suw+tvx+twx}{stwx}$} (8,-3) node{$1$} ;   
          \draw (0,-4) node{$1$} (2,-4) node{$\frac{sux+tw+tx+uw}{uwx}$} (4,-4) node{$Z$} (6,-4) node{$\frac{suv+sw+tv+tw}{svw}$} (8,-4) node{$1$} ;
          \draw (0,-5) node{$1$} (2,-5) node{$\frac{t+u}{x}$} (4,-5) node{$\frac{su+t}{w}$} (6,-5) node{$\frac{s+t}{v}$} (8,-5) node{$1$} ; 
          \draw (0,-6) node{$1$} (2,-6) node{$u$} (4,-6) node{$t$} (6,-6) node{$s$} (8,-6) node{$1$} ;
          \draw (0,-7) node{$1$} (2,-7) node{$x$} (4,-7) node{$w$} (6,-7) node{$v$} (8,-7) node{$1$} ;       
        \end{tikzpicture}
      \end{center}
where 
\[
 X=\frac{suvx+suw+tvw+tvx+tw^2+twx}{stuw}, \\
\]
\[
 Y =    \frac{stuvx+stw^2+stwx+suvwx+suw^2+t^2vw+t^2vx+t^2w^2+t^2wx+tuvw+tuw^2}{stuvwx}, \quad \textrm{and}
\]
\[
 Z = \frac{stw+suvx+suw+t^2w+tuw+tvx}{tvwx}.
\]

If, instead, we fix the six first entries of the leftmost non-trivial column, then we get:
      \begin{center}
        \begin{tikzpicture}
          \draw (0,1) node{$1$} (2.5,1) node{$s$} (5,1) node{$\frac{sux-s-tx+t}{tvx-tw-ux+1}$} (7.5,1) node{$B$} (10,1) node{$1$} ;
          \draw (0,0) node{$1$} (2.5,0) node{$t$} (5,0) node{$su-t$} (7.5,0) node{$\frac{suw-sv+tv-tw-u+1}{tvx-tw-ux+1}$} (10,0) node{$1$} ;
          \draw (0,-1) node{$1$} (2.5,-1) node{$u$} (5,-1) node{$tv-u$} (7.5,-1) node{$suw-sv-tw+1$} (10,-1) node{$1$} ;
          \draw (0,-2) node{$1$} (2.5,-2) node{$v$} (5,-2) node{$uw-v$} (7.5,-2) node{$tvx-tw-ux+1$} (10,-2) node{$1$} ;
          \draw (0,-3) node{$1$} (2.5,-3) node{$w$} (5,-3) node{$vx-w$} (7.5,-3) node{$\frac{tvx-tw+uw-ux-v+1}{suw-sv-tw+1}$} (10,-3) node{$1$} ;   
          \draw (0,-4) node{$1$} (2.5,-4) node{$x$} (5,-4) node{$\frac{svw-sw+w-x}{suw-sv-tw+1}$} (7.5,-4) node{$C$} (10,-4) node{$1$} ;
          \draw (0,-5) node{$1$} (2.5,-5) node{$\frac{sux-s-tx+1}{suw-sv-tw+1}$} (5,-5) node{$A$} (7.5,-5) node{$\frac{svx-sw-x+1}{tvx-tw-ux+1}$} (10,-5) node{$1$} ;  
          \draw (0,-6) node{$1$} (2.5,-6) node{$B$} (5,-6) node{$\frac{sux-s-tx+t}{tvx-tw-ux+1}$} (7.5,-6) node{$s$} (10,-6) node{$1$} ;
          \draw (0,-7) node{$1$} (2.5,-7) node{$\frac{suw-sv+tv-tw-u+1}{tvx-tw-ux+1}$} (5,-7) node{$su-t$} (7.5,-7) node{$t$} (10,-7) node{$1$} ;      
          \draw (0,-8) node{$1$} (2.5,-8) node{$suv-sv-tw+1$} (5,-8) 
          node{$tv-u$} (7.5,-8) node{$u$} (10,-8) node{$1$} ;       
        \end{tikzpicture}
      \end{center}
where 
\small
\[
A=\frac{stvx-stw+suwx-sux-svx+s-twx+tw+x-1}{stuvwx-stuw^2-stv^2x+stvw-su^2wx+suvx+suw-sv-t^2vwx+t^2w^2+tuwx+tvx-2tw-ux+1},
\]
\normalsize
\small
\[
 B=\frac{stuvx-stuw-sr^2x+suw+su-sv-t^2vx+t^2w+tux+tv-tw-t-u+1}{stuvwx-stuw^2-stv^2x+stvw-su^2wx+suvx+suw-sv-t^2vwx+t^2w^2+tuwx+tvx-2tw-ux+1},
\]
\normalsize
\small
\[
 C=\frac{suvwx-suw^2-sv- 2x+svw-tvwx+tvx+tw^2-tw+uw-ux+vx-v-w+1}{stuvwx-stuw^2-stv^2x+stvw-su^2wx+suvx+suw-sv-t^2vwx+t^2w^2+tuwx+tvx-2tw-ux+1}.
\]
\normalsize

\section*{B.3. Finiteness of the number of $2$-friezes of height $3$}
We use the second choice of initial variables of the previous section.  We can assume, without loss of generality, that the greatest entry in the first and third columns is $u$.  Then

\begin{eqnarray*}
 u & \geq & suw-sv-tw+1 \\
   & \geq & suw-su-wu+1 \textrm{ (since $u\geq t,v$)} \\ 
   & = & u(s-1)(w-1)-u + 1.
\end{eqnarray*}

Therefore 
\begin{eqnarray*}
 1 & \geq & (s-1)(w-1) - 1 +\frac{1}{u} \\
   & > &  (s-1)(w-1) - 1.  
\end{eqnarray*}

Hence $(s-1)(w-1) < 2$.  This implies that $s=1$ or $w=1$ or $s=w=2$.

\subsection*{B.3.1 The case $s=1$ or $w=1$}
If $s=1$ or $w=1$, then the associated frieze of type $E_6$ contains a frieze of type $D_5$, and it is known \cite{FP16} that there are only $187$ of these.  Thus there is a finite number of cases where $s=1$ or $w=1$.

\subsection*{B.3.2 The case $s=w=2$}
If $s=w=2$, then consider the following inequalities:
\begin{eqnarray*}
u & \geq & suw-sv-tw+1 \\
  & =    & 4u-2t-2v+1 
\end{eqnarray*}
and~$u\geq v$, which implies that $3u \geq 4u-2t+1$, so
\[ u \le 2t-1. \]
But together with $u \ge tvx-tw-ux+1 = tvx-2t-ux+1$ this yields
\[ (2t-1)(x+1) \ge tvx-2t+1, \]
and hence
\[ 4t \ge tvx+2-2tx+x = x+2+(v-2)tx>(v-2)tx.\]
Thus we obtain $(v-2)x<4$ which gives $v< 6$. 
For symmetry reasons, the same argument produces $t< 6$. 
But then $u\le 9$ since $u\le 2t-1$. 
Hence we have reduced the problem to a finite number of cases. In fact, an easy 
computation shows that the only solution is
\[ (s,t,u,v,w,x) = (2, 4, 5, 4, 2, 1) \]
in which case the (transposed) 2-frieze is:
\begin{center}
\begin{tabular}{cccccccccccccc}
1 & 1 & 1 & 1 & 1 & 1 & 1 & 1 & 1 & 1 & 1 & 1 & 1 & 1 \\
1 & 2 & 4 & 5 & 4 & 2 & 1 & 1 & 2 & 4 & 5 & 4 & 2 & 1 \\
1 & 2 & 6 & 11 & 6 & 2 & 1 & 1 & 2 & 6 & 11 & 6 & 2 & 1 \\
1 & 2 & 4 & 5 & 4 & 2 & 1 & 1 & 2 & 4 & 5 & 4 & 2 & 1 \\
1 & 1 & 1 & 1 & 1 & 1 & 1 & 1 & 1 & 1 & 1 & 1 & 1 & 1
\end{tabular}
\end{center}

\medskip

Thus the number of~$2$-friezes of height~$3$ is finite, and we know a bound on the values appearing in such a~$2$-frieze.  
A computer check then allows to show that there are only~$868$ such~$2$-friezes.
Moreover, by~\cite{MGOT12}, the number of friezes of type~$E_6$ is at most the number of~$2$-friezes of height~$3$. 
In fact, the two numbers are equal; this follows 
from Corollary~\ref{cor:mesh-SLk}; since 
$2$-friezes are in bijection with SL$_3$-friezes, \cite[Section 3]{MGOT12}. 
Since we know from~\cite{Propp, MGOT12, FP16} that this number is at least~$868$, we have thus proved that the number is exactly~$868$.
This finishes the proof of Theorem~\ref{theo::main}.

\noindent
\begin{footnotesize}
$^{\#}$ {\sc Leibniz Universit\"at Hannover, Institut f\"ur Algebra, Zahlentheorie und diskrete Mathematik, Fakult\"at f\"ur Mathematik 
und Physik, Welfengarten 1, D-30167 Hannover, Germany}\\
{\em E-mail address}: \email{cuntz@math.uni-hannover.de}
\vskip.3cm
$^{\dagger}$ {\sc Laboratoire de Math\'ematiques d'Orsay, Universit\'e Paris-Sud, CNRS, Universit\'e Paris-Saclay, 91405 Orsay, France} \\
{\em E-mail address}: \email{pierre-guy.plamondon@math.u-psud.fr}\\
\end{footnotesize}

\bibliographystyle{alpha}
\bibliography{biblio}

\end{document}